\documentclass[12pt]{amsart} 
\usepackage{amssymb}
\usepackage{amsmath}
\usepackage{amsthm}
\usepackage{natbib}
\usepackage{pstricks}
\usepackage{pst-node}
\usepackage[normalem]{ulem}
\usepackage{bm} 
\usepackage{graphicx}
\usepackage{color}
\def\red#1 {\textcolor{red}{#1 }}

\newtheorem{thm}{Theorem}[section]
\newtheorem{lem}[thm]{Lemma}
\newtheorem{prop}[thm]{Proposition}
\newtheorem{cor}[thm]{Corollary}
\theoremstyle{remark}
    \newtheorem{rem}{Remark}   
\theoremstyle{definition}
    \newtheorem{defn}[thm]{Definition}

\numberwithin{thm}{section}    
\allowdisplaybreaks[4]  

\def\Z {{\mathbb Z}}
\def\N {\mathbb N}
\def\T {{\tilde T}}
\def\H {{\mathcal H}}

\def\B {{\mathcal B}}

\def\l {\left\langle}
\def\r {\right\rangle}

\def\cx {{\mathbb C}}
\def\al {{\alpha}}

\begin{document}

\title[centralizers in Hecke algebras]{Centralizers in the Hecke algebras of complex reflection groups.}
\author{Andrew Francis}
\address{School of Computing and Mathematics\\ University of Western Sydney\\ Locked Bag 1797, South Penrith DC, NSW 1797\\ Australia}
\email{a.francis@uws.edu.au}

\begin{abstract}
How far can the elementary description of centralizers of parabolic subalgebras of Hecke algebras of finite real reflection groups be generalized to the complex reflection group case?  In this paper we begin to answer this question by establishing results in two directions.  First, under conditions closely analogous to those existing for the real case, we give explicit relations between coefficients in an element centralizing a generator.  Second, we introduce a tool for dealing with a major challenge of the complex case --- the ``instability'' of certain double cosets --- through the definition and use of a double coset graph.  We use these results to find integral bases for the centralizers of generators as well as the centres of the Hecke algebras of types $G_4$ and $G(4,1,2)$.

\end{abstract}

\maketitle

\setcounter{section}{-1}

\section{Introduction}

The purpose of this paper is to begin a study of the integral centres and centralizers of Hecke algebras of complex reflection groups.  In the case of Hecke algebras of real reflection groups, these subalgebras are fairly well understood.  For instance, it is possible to algorithmically construct integral bases for centralizers of parabolic subalgebras of the Hecke algebra that are analogous to class-sum bases of centralizers in the group algebra~\cite{GR97,Fmb,Fcent,Fcent2}.
From an alternative viewpoint, it is known that the centre of the Hecke algebra of a real reflection group is integrally the set of symmetric functions of Jucys-Murphy elements~\cite{DJ87,FG:DJconj2006}. The corresponding result for Hecke algebras of complex reflection groups is not known, although it does hold when the algebra is semisimple, at least for the Ariki-Koike algebras~\cite{RamRamagge03}.

This paper generalizes aspects of the former of the above approaches.  One of the ingredients in our understanding of the real reflection group case is the existence of an integral basis for the centralizer of a standard generator of the Hecke algebra with nice properties~\cite{DJ87,Fmb}.  Finding an integral basis for the centralizer of a generator in the complex reflection group situation proves to be much more complicated.  The obstacles are mainly due to problems associated with the lack of a basis for the algebra which is independent of choices for reduced expressions of words in the complex reflection group.   That is, in the case of complex reflection groups, it is \emph{not} the case that if $s_{i_1}\dots s_{i_t}$ and $s_{j_1}\dots s_{j_t}$ are reduced expressions for the same element of the group, then the corresponding elements in the Hecke algebra $T_{s_{i_1}}\dots T_{s_{i_t}}$ and $T_{s_{j_1}}\dots T_{s_{j_t}}$ are also equal.  Other related problems are to do with the double cosets in the complex reflection groups, which in general fail to satisfy a ``stability'' property that holds in the real case.

In this paper we find relationships between coefficients in elements of the centralizer of a generator whenever the double cosets behave as they do in the real case. This allows us to find integral bases for the centralizers under these conditions.  In the Hecke algebra of type $G_4$, for example, this behaviour is seen for all double cosets, and this allows us to give a basis for the centraliser of each generator, and in turn to find a basis for the centre.

To study the centralizer of a generator when the double cosets are not ``stable'' (Definition~\ref{def:stabledcoset}), we introduce a relation among the double cosets that restricts the problem somewhat.  The ``$\H$-double coset graph'' of a complex reflection group is fully described for the class of groups $G(r,1,2)$ in Section \ref{sec:gr1n:reduced.exprs}, and we use this to give bases for the centralizers of generators for the Hecke algebra of $G(4,1,2)$, in turn using these bases to find an integral basis for the centre of the Hecke algebra of $G(4,1,2)$.

The paper is organized as follows.  Section \ref{sec:background} gives the basic definitions of complex reflection groups and their Hecke algebras, including a definition of normalized generators which we introduce to make the exposition simpler.  Section \ref{sec:d.cosets.in.cx.ref.gps} introduces the notion of the $\H$-double coset graph, used later to find relationships among coefficients in the centralizer of a generator.  We deal with the case where the double cosets are well-behaved (``stable'') in Section \ref{sec:additive.d.cosets.coeff.rels}, a case which despite having behaviour analogous to the real case still poses a technical challenge to resolve.  Theorem \ref{characterized} gives the precise relationship between coefficients in an element centralizing a generator in this case.  In Section \ref{sec:basis.thm.for.centralizers} a result is proved giving sufficient conditions for the existence of an integral basis for the centralizer of a parabolic subalgebra.  This result (Proposition \ref{basis}) summarizes the approach taken in other papers \cite{Fmb,Fcent,Fcent2} to finding integral bases for centralizers in more amenable (Coxeter group) situations.  The principal being applied in this paper is to attempt to develop the correct analogy to this approach.

In Section \ref{sec:g4} we apply some of the earlier results to the group $G_4$, studying firstly the centralizers and then the centre, giving integral bases in each case.  Finally Section \ref{sec:Gr1n} addresses the Hecke algebras of the groups $G(r,1,n)$.  Section \ref{sec:gr1n:reduced.exprs} gives some results on reduced expressions specific to these groups, and a general statement about some of the double cosets that are easier to deal with.  Section \ref{sec:gr12.d.coset.graphs} provides a complete description of the $\H$-double coset graphs for the Hecke algebras of $G(r,1,2)$.  Section \ref{sec:g412.bases} gives the example of $G(4,1,2)$, demonstrating bases for the centralizer of each generator as well as for the centre.

I would like to thank Michel Brou\'e for suggesting in December 2000 that I look at the centre of the Hecke algebra of $G_4$, a question which led to the work in this paper, and Chi Mak, who gave me some helpful comments on an earlier version of this paper.

\section{Preliminaries}\label{sec:background}

\subsection{Complex reflection groups}\label{sec:bg:cxrefgps}
Groups of linear transformations which act on the vector space $\cx^n$ fixing a hyperplane are called complex reflection groups.  The finite complex reflection groups were characterized in 1954 by Shephard and Todd \cite{ShTo54} into an infinite family $G(r,p,n)$ where $p|r$, and 34 exceptional groups labelled $G_4$ to $G_{37}$.  These groups include the finite real reflection groups (Coxeter groups), and can be given Coxeter-type presentations (\cite{BMR98,BessisMichel2003}) consisting of a generating set of ``pseudo-reflections'' accompanied by a set of homogeneous relations.  The pseudo-reflections are reflections in the sense that they fix a hyperplane in $\cx^n$, but are not involutions as their non-trivial eigenvalue may be a root of unity other than $-1$.

For example, the infinite family $G(r,1,n)$, which includes the Coxeter groups of types $A$ and $B$ (when $r=1$ and 2 respectively), has generating set $S=\{t,s_1,\dots,s_{n-1}\}$ with order relations $t^r=s_i^2=1$, and homogeneous relations $ts_1ts_1=s_1ts_1t$, $s_is_{i+1}s_i=s_{i+1}s_is_{i+1}$ for $0<i<n-1$ and $s_is_j=s_js_i$ if $|i-j|>1$.  It can be represented by a Coxeter-type diagram as in Figure \ref{fig:gr1n}.

\begin{figure}[ht]
   \begin{minipage}[t]{\linewidth}
        \centering
        \includegraphics{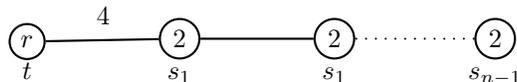}
\caption{Coxeter-type diagram for the complex reflection group $G(r,1,n)$}\label{fig:gr1n}
   \end{minipage}%
\end{figure}

Let $W$ be a finite complex reflection group with generating set $S$.  For $g\in W$, we set the length $l(g)$ of $g$ to be the minimal number of standard generators required to write $g$ (with positive exponents).

\subsection{The Hecke algebra of a complex reflection group}\label{sec:heckealg}

\subsubsection{Standard definitions}
The group algebra of a complex reflection group $W$ can be viewed as a quotient of a corresponding braid group algebra by an ideal generated by the order relations in the group.  Hecke algebras of the complex reflection groups are also obtained by taking a quotient of the braid group algebra (considered over a ring of polynomials in some indeterminates), but by an ideal generated by deformed order relations.  The Hecke algebras are then so-called ``order-deformations'' of the group algebra of the complex reflection group.  Full details of the construction of these algebras in this way are found in \cite{BMR98} and \cite{BrBoston}.

For each generator $s\in S$ of order $m_s$, we define a set of indeterminates $\{\al_{s,i}\}$ for $i=0,\dots,m_s-1$, and set $\al_{s,i}=\al_{t,i}$ if $s$ and $t$ are conjugate pseudo-reflections.  Let $\bm\al=\{\al_{s,i}^{1/m_s}\mid s\in S,0\le i<m_s\}$ and let $\bm\al^{-1}$ be the set of inverses of the parameters $\bm\al$.  Let $R_\al=\Z[\bm\al,\bm\al^{-1}]$.  We will write $T_s$ for the generator of the Hecke algebra $\H_\al=\H_\al(W)$ corresponding to the generator $s$ of the braid group and of the complex reflection group.  The Hecke algebra of the complex reflection group $W$ is then defined over $R_\al$ to be the quotient of the braid group algebra by the deformed order relations
\[\prod_{i=0}^{m_s-1}(T_s-\al_{s,i})=0, \quad\text{for }s\in S.\]
Thus $\H_\al$ inherits the defining homogeneous relations from the group algebra of the complex reflection group, but instead of order relations $s^{m_s}=1$, $T_s^{m_s}$ is a polynomial in $T_s$ of degree $m_s-1$ and whose coefficients are elementary symmetric functions in $\bm\al$.  If we set $e_r=e_r(\al_{s,0},\dots,\al_{s,m_s-1})$ to be the $r$th elementary symmetric function in the $\al_{s,i}$ then
\begin{equation}\label{powerrel}
T_s^{m_s}=e_1T_{s^{m_s-1}}+\dots +(-1)^{i-1}e_i T_{s^{m_s-i}}+\dots+ (-1)^{m_s-1}e_{m_s}T_1.
\end{equation}
Note that the specialization $\rho:\al_{s,j}\mapsto \omega^j$ (for $\omega$ a primitive $m_s$th root of unity) sends each elementary symmetric function $e_i$ to zero except for $e_{m_s}$ which is sent to 1 if $m_s$ is odd and $-1$ if $m_s$ is even.  This specialization thus sends relation \eqref{powerrel} to the group algebra relation $T_s^{m_s}=1$.

\subsubsection{Normalizing coefficients}
We will normalise the coefficients and generators in $\H_\al$, instead considering an algebra $\H$ over a subring $R$ of $\bar R_\al:=R_\al[\zeta]$ where $\zeta$ is a primitive $2m_s$th root of unity.  Then $\H$ is a subalgebra of $\bar R_\al\otimes\H_\al$.
Results obtained in $\H$ have corresponding versions in the full Hecke algebra $\H_\al$.

For $0\le i<m_s$, we set $\T_1=T_1$ and
\begin{align*}
\xi_{s,i} &=-(-e_{m_s})^{\frac{i-m_s}{m_s}}e_{m_s-i} \quad\text{ and}\\
 \T_s    &=-(-e_{m_s})^{-1/m_s}T_s.
\end{align*}
Note that $e_{m_s}=\al_{s,0}\al_{s,1}\dots\al_{s,m_s-1}$ is invertible in $R_\al$, and $\xi_{s,0}$ is identically one.  Under this change of parameterization, the order relation \eqref{powerrel} becomes
\begin{equation}\label{powerrel:xi}
\T_s^{m_s}=\xi_{s,m_s-1}\T_{s^{m_s-1}}+\dots +\xi_{s,m_s-i}\T_{s^{m_s-i}}+\dots+\T_1.
\end{equation}
Let $\bm\xi:=\{\xi_{s,i}\mid s\in S, 0\le i< m_s\}$.

Let $\H$ denote the algebra generated by $\{\T_s\mid s\in S\}$ over $R=\Z[\bm\xi]$ subject to the inherited braid relations and the normalized order relation \eqref{powerrel:xi}.  Note that when the indeterminates $\al_{s,i}$ are specialized to roots of unity via $\rho$, we have
$\xi_{s,i}\mapsto 0 $ for all $0<i<m_s$.

The following will be of use later in the paper.

\begin{lem}\label{lem:power.of.Ts}
  Let $m$ be the order of the generator $s$ and write $\xi_i=\xi_{s,i}$. For $k\ge 0$, we have
\[\T_s^{m+k}=\sum_{i=0}^{m-1}a_i\T_{s^i}\]
where the coefficients $a_i$ are polynomials in $\{\xi_j\mid 0\le j\le m-1\}$ such that the index sum of each monomial in $a_i$ is congruent to $i-k\mod m$.  In particular, the coefficients $a_i$ are distinct.
\end{lem}
\begin{proof}
By induction on $k$.  If $k=0$ then we have from \eqref{powerrel:xi}:
  \[\T_s^m=\xi_0\T_1+\xi_1\T_{s}+\dots+\xi_{m-1}\T_{s^{m-1}}.\]
That is, $a_i=\xi_i$ and the index sum of the only term in $a_i$ is $i$.  Since $k=0$ the statement holds.

  Suppose now that $\T_s^{m+k-1}=\sum_{i=0}^{m-1}a_i\T_{s^i}$ with each $a_i$ being a linear combination of monomials in the $\xi_j$ with index sum $\equiv i-(k-1)\mod m$.  Then (setting $a_{-1}=0$)
\begin{align*}
\T_s^{m+k}=\T_s\sum_{i=0}^{m-1}a_i\T_{s^i}
    &=\left(\sum_{i=0}^{m-2}a_i\T_{s^{i+1}}\right)+a_{m-1}\T_s^m\\
    &=\sum_{i=1}^{m-1}a_{i-1}\T_{s^i}+\sum_{i=0}^{m-1}a_{m-1}\xi_i\T_{s^i}\\
    &=\sum_{i=0}^{m-1}\left(a_{i-1}+a_{m-1}\xi_i\right)\T_{s^i}.
\end{align*}
By induction we have that the index sums of $a_{i-1}$ are congruent to $(i-1)-(k-1)\equiv i-k\mod m$.  Similarly, since the index sum of $a_{m-1}$ is congruent to $(m-1)-(k-1)\equiv m-k\mod m$, we have that the index sum of $a_{m-1}\xi_i$ is congruent to $i+(m-k)\equiv i-k\mod m$.  Thus $a_{i-1}+a_{m-1}\xi_i$ has monomials whose index sums are congruent to $i-k\mod m$.
\end{proof}

\subsubsection{The algebra as an $R$-module}\label{subsubsec:R-mod}

It has been conjectured that the Hecke algebra $\H$ of a finite complex reflection group $W$ is a free $R$-module of dimension $|W|$ (see \cite[(1.17)]{BMM99} or \cite[(4.23)]{BMR98}).  This is known for the infinite family $G(r,p,n)$ and for all but a handful of exceptional cases (\cite{Mueller03}).  It is a slightly different question as to whether $\H$ has a particularly nice basis (for instance a basis consisting of reduced words in the generators $S$), whose answer is less complete (op. cit.).

The key problem for dealing with the Hecke algebra of a complex reflection group as an $R$-module is the lack of a natural choice of basis corresponding to the standard basis of the corresponding group algebra (the set of group elements).  This lack of a natural basis is caused by problems with reduced expressions for elements of the complex reflection group.

Let $Red(S)$ be the set of reduced words in the generating set $S$ for the complex reflection group $W$. Let $\sim_\mathrm{br}$ denote the equivalence relation on $Red(S)$ generated by the braid relations of $W$.  Then the map $\phi:Red(S)/\sim_{\rm br}\to W$ defined by $\phi(s)=s$ for $s\in S$ is surjective but not, in general, injective.  That is, there may be several $\sim_{\rm br}$-equivalence classes of reduced words in $S$ for a given element of $W$.

For each map $\eta:W\to Red(S)/\sim_{\rm br}$ satisfying $\phi_\circ\eta=\iota_W$, the identity map on $W$, we obtain a set of equivalence classes of reduced words $\{\eta(w)\mid w\in W\}$.  We will use the bold faced $\bm w$ to denote a representative of the class $\eta(w)$.  That is, $\bm w$ is a reduced word in the elements of $S$ representing the group element $w$ (following the conventions of~\cite{BremMa97}).
If $\bm w=s_{i_1}\dots s_{i_k}\in Red(S)$, then we write $\T_{\bm w}=\T_{s_{i_1}\dots s_{i_k}}:=\T_{s_{i_1}}\dots\T_{s_{i_k}}$.  If $\bm w\sim_{\rm br}\bm w'$ then $\T_{\bm w}=\T_{\bm w'}$.

In what follows we will assume that $\H$ is a free $R$-module with a basis $B$ of reduced words representing the elements of $W$ (a \emph{reduced basis} of $\H$).  A choice of reduced basis requires a choice of reduced expression for each element of the group $W$, or at least a choice of equivalence class in $Red(S)/\sim_{\rm br}$.  Consequently, each reduced basis $B$ corresponds to a map $\eta_B:W\to Red(S)/\sim_{\rm br}$.  To emphasize that the reduced word representing a group element is dependent on the basis $B$ via $\eta_B$, we write $\bm w, B$ for a representative of $\eta_B(w)$, so that $B=\{\T_{\bm w, B}\mid w\in W\}$.  That is, $\T_{\bm w,B}$ denotes the Hecke algebra element in the basis $B$ for $\H$ that corresponds to the group element $w$ (via the reduced expression $\bm w$ for $w$).

The assumption of the existence of a reduced basis, while known to hold for most cases, is not in itself enough to ensure that $\H$ behaves well (in the sense that it ``behaves like the Hecke algebra of a Coxeter group'').  With Hecke algebras of Coxeter groups, bases of reduced words of group elements are in fact equal, as each pair of reduced words for the same group element can be transformed from one to the other using only the braid relations.
That is, for Coxeter groups, the map $\phi:Red(S)/\sim_{\rm br}\to W$ is bijective; this ensures that the corresponding Hecke algebra elements are equal.  In Hecke algebras of complex reflection groups it is not in general the case that $\phi$ is injective and therefore not in general the case that two Hecke algebra elements corresponding to reduced words for the same group element are equal.  Hence there may be several distinct reduced bases for a given Hecke algebra.  In what follows we shall see that some (reduced) bases are more suited than others to different needs.

Although different reduced words may give different Hecke algebra elements, it is possible to say something about the relationship between the Hecke algebra elements in some complex reflection groups.

\begin{prop}[\cite{BremMa97}]\label{prop:bremal97}
Let $W=G(r,1,n)$.  If $\bm w$ and $\bm w'$ are two reduced words for the same element $w$ of $W$, then
\[\T_{\bm w}-\T_{\bm w'}\in\sum_{y\not\in S_n,l(y)<l(w)}R\T_{\bm y}.\]
\end{prop}

\subsection{Double cosets in complex reflection groups}\label{sec:d.cosets.in.cx.ref.gps}

An important component of what follows will be the behaviour of the span of certain double cosets within the Hecke algebra.

For $J\subset S$, the parabolic subgroup $W_J$ of a complex reflection group $W$ is generated by the elements of $J$ subject to the relevant relations from $W$.  Let $\H$ be the Hecke algebra of $W$ with reduced basis $B$.  Write $\H_J$ for the parabolic subalgebra of $\H$ corresponding to $W_J$, and write $\H_{W_JdW_J}$ for the submodule of $\H$ spanned over $R$ by the set $\{\T_{\bm w, B}\mid w\in W_JdW_J\}$.
It should be emphasized that $\H_{W_JdW_J}$ differs with respect to different bases of $\H$.

\begin{defn}\label{def:stablebasis}\label{def:stabledcoset}
  Let $B$ be a reduced basis for $\H=\H(W)$, let $W_J$ be a parabolic subgroup of $W$, and let $d\in W$.

  We say the double coset $W_JdW_J$ is \emph{stable with respect to $B$} if
    \[\H_J\T_{\bm d, B}\H_J\subseteq\H_{W_JdW_J}.\]

  We say the reduced basis $B$ is \emph{$W_J$-stable} if \emph{all} double cosets $W_JdW_J$ are stable with respect to $B$.  If $W_J=\l s\r$ for some $s\in S$, we abbreviate this to \emph{$s$-stable}.
\end{defn}

These definitions are trivial in the case $W$ is a Coxeter group as (i) all reduced bases $B$ for $\H$ are identical, (ii) all double cosets of parabolic subgroups are stable with respect to any $B$, and therefore (iii) all reduced bases (and there is only one) are $W_J$-stable for any $J\subseteq S$.  We will see examples of such double cosets and such bases in later sections (Sections \ref{sec:g4.hecke.bases} and \ref{sec:gr1n:reduced.exprs}).  In general, $W_J$-stable bases seem to be the exception rather than the rule.

In finite Coxeter groups one has some additional useful properties of double cosets.  Each double coset contains a unique element $d$ of minimal length, called a \emph{distinguished} element.  In this case every element $w$ of the double coset $W_JdW_J$ can be written $w=w_1dw_2$ for some $w_i\in W_J$ and with $l(w)=l(d)+l(w_1)+l(w_2)$.  There are tighter facts about this due to Howlett \cite[Prop.~(2.7.5)]{Carter85}.  The double cosets are then neatly divided into those that centralize $W_J$ and those that don't.  This partition carries over to the Hecke algebras of finite Coxeter groups: if a $W_J$-$W_J$ double coset representative $d$ is in the centralizer $Z_W(W_J)$ of $W_J$ in $W$, then $\T_{d}\in Z_\H(\H_J)$.

In the generality of complex reflection groups, it is no longer the case that double cosets always have unique minimal length elements.  Even when they do, it is not always the case that if a (unique) minimal length element $d\in Z_W(W_J)$ then $\T_{\bm d,B}\in Z_\H(\H_J)$.  As far as I am aware no results analogous to those of Howlett are known for complex reflection groups.

\subsection{The $\H$-double coset graph}\label{sec:bg:H.dcoset.graph}

Let $W$ be a complex reflection group with Hecke algebra $\H$ defined over $R$ with reduced basis $B$.  We define a directed graph $\mathcal G$ whose vertices are the double cosets $W_JdW_J$ and with arrows
\[W_JdW_J\longrightarrow W_Jd'W_J\]
\text{when }
\[\H_J\T_{\bm d,B}\H_J\cap\H_{W_Jd'W_J}\neq\emptyset.\]
A \emph{terminal} vertex $W_JdW_J$ is one for which there is no $W_Jd'W_J\neq W_JdW_J$ such that $W_JdW_J\to W_Jd'W_J$.

Note:
\begin{itemize}
    \item For each $d\in W$, $W_JdW_J\to W_JdW_J$.
    \item A double coset is stable if and only if it is a terminal vertex of $\mathcal G$.
    \item In a finite Coxeter group all vertices are terminal as all double cosets are stable.
\end{itemize}
In this paper we will focus on the case $|J|=1$.  

The $\H$-double coset graph will be a useful way to visualize the relationships among double cosets when it comes to determining relationships among coefficients in an element of the centralizer $Z_\H(\H_J)$.  This is because it will be enough to consider, for each double coset $W_JdW_J$, relationships between coefficients $r_{\bm w,B}$ in
\[\sum_{\substack{w\in W_Jd'W_J,\\W_JdW_J\to W_Jd'W_J}}r_{\bm w,B}\T_{\bm w,B}.\]

\section{The centralizer of a generator: additive double cosets}\label{sec:additive.d.cosets.coeff.rels}

In the current section we focus on the centralizer of a generator in the Hecke algebra of a complex reflection group: the case $J=\{s\}\subseteq S$.  The reason for this is that we hope that by understanding bases for centralizers of generators we may build bases for centralizers of larger subalgebras (including the centre).  This is the philosophy used for Coxeter groups in \cite{Fmb,Fcent,Fcent2}, and we apply it in some cases in Sections~\ref{sec:g4} and~\ref{sec:Gr1n}.

Let $\H$ be the Hecke algebra of a complex reflection group $W$, with reduced basis $B$.

\subsection{Additivity and stability}

\begin{defn}\label{def:additive} An $\l s\r$-$\l s\r$ double coset is said to be \emph{additive} if it contains a minimal length representative $d$ with the property that $l(s^ids^j)=l(d)+i+j$ for $0\le i,j\le m_s-1$.

An $\l s\r$-$\l s\r$ double coset $\l s\r d\l s\r$ is said to be \emph{centralizing} if for all $w\in \l s\r d\l s\r$, $\T_s\T_{\bm w,B}=\T_{\bm w,B}\T_s$.
\end{defn}

The additivity as defined is not a useful definition for larger parabolic subgroups $W_J$ for $J\subseteq S$ and $|J|>1$.  In fact, in this case even Coxeter groups do not have many additive double cosets (see the remarks above about the result of Howlett).

Unlike the situation for finite Coxeter groups, these are not the only possibilities for $\l s\r$-$\l s\r$ double cosets in $W$.  For instance, in the case of complex reflection groups we do not have that $d\in Z_W(s)\implies \T_{\bm d,B}\in Z_\H(s)$.  Also, if $d\not\in Z_W(W_J)$ it does not imply that $\l s\r d\l s\r$ is additive.

However, despite the extra classes of $\l s\r$-$\l s\r$ double cosets --- those that are neither additive nor centralizing --- the additive double cosets nevertheless form a significant case.  For instance, in the infinite family $G(r,1,n)$ there are a total of $r^{n-2}(n-1+r)(n-1)!$  $\l t\r$-$\l t\r$ double cosets, of which $r^{n-2}(n-1)(n-1)!$ are additive~\cite{FM:dcosets}.

If the double coset $\l s\r d\l s\r$ is additive, then we have an explicit listing of its elements
\begin{equation}\l s\r d\l s\r=\{s^ids^j\mid 0\le i,j<m_s\}\label{eq:additive_d_cosets}\end{equation}
and of course by definition of additivity, $l(s^ids^j)=l(d)+i+j$ for each $0\le i,j\le m_s-1$.  Thus such a double coset has exactly $m_s^2$ elements.

If the double coset $\l s\r d\l s\r$ is centralizing then $d\in Z_W(s)$ and
\[\l s\r d\l s\r=\{ds^i\mid 0\le i< m_s\}.\]

\begin{lem}\label{lem:additive_implies_stable}
  Let $\H$ be the Hecke algebra of the complex reflection group $G$, with reduced basis $B$, and let $\l s\r d\l s\r$ be a double coset in $G$ for $s\in S$.

  If $\l s\r d\l s\r$ is additive and
  $\T_{s^i}\T_{\bm d,B}\T_{s^j}=\T_{\bm{s^ids^j},B}$ for any $1\le i,j\le m_s-1$,
  then $\l s\r d\l s\r$ is stable with respect to $B$.
\end{lem}

\begin{proof}
  This is clear from the definitions and from \eqref{eq:additive_d_cosets}.
\end{proof}

Given a fixed reduced basis for the Hecke algebra, the basic tactic is to find relations among the coefficients of these basis elements in an element of the centralizer.  That is, if $B$ is a reduced basis for $\H$, then $h\in Z_\H(\T_s)$ may be written as an $R$-linear combination $h=\sum_{w\in W}r_{\bm w,B}\T_{\bm w,B}$.  Our question is, what relationships among coefficients $r_{\bm w,B}$ are necessary or sufficient for $h$ to centralize $\T_s$?

\subsection{Coefficient relations}

In this subsection we provide some relationships in the case that the $\l s\r$-$\l s\r$ double coset is stable and additive.  We start by proving a Lemma that gives a symmetry feature of centralizer elements in common with those of Hecke algebras of Coxeter groups.

\begin{lem}\label{symmetry} 
Suppose that $\H$ is a Hecke algebra of a complex reflection group with basis $B$ satisfying $\T_{s^i}\T_{\bm d,B}\T_{s^j}=\T_{\bm{s^ids^j},B}$ for any $1\le i,j\le m_s-1$, where $\l s\r d\l s\r$ is additive with $d$ its minimal length representative.  Write $n=m_s$.

If $h=\sum_{w\in W} r_{\bm w,B}\T_{\bm w,B}\in Z_\H(\T_s)$, then
\begin{equation*}
r_{\bm{s^ids^j},B}=r_{\bm{s^jds^i},B}
\end{equation*}
for any $0\le i,j\le n-1$.
\end{lem}
\begin{proof}
We will prove this by induction, treating two separate cases simultaneously.  The general approach is to check the coefficients of an element $\T_{\bm w,B}$ in both $\T_sh$ and $h\T_s$.  These should of course be equal, since $h$ commutes with $\T_s$.

An important part of the approach is to restrict our attention to the $\l s\r$-$\l s\r$ double coset containing $d$.  This is possible because the double coset is stable with respect to the basis $B$ of $\H$ (Lemma~\ref{lem:additive_implies_stable}).  This restriction allows us to consider instead of the whole of $h=\sum_{w\in W}r_{\bm w,B}\T_{\bm w,B}$ just the sum
\[\sum_{w\in \l s\r d\l s\r}r_{\bm w,B}\T_{\bm w,B}\]
of terms from the double coset, since the product of $\T_s$ with any of these terms is a linear combination of terms from within the same double coset (see the definition of a stable basis Definition \ref{def:stablebasis}).  This idea was used in the symmetric group case by R.~Dipper and G.~James \cite[(2.4)]{DJ87}, and in the finite Coxeter group case in \cite[(3.1)]{Fmb}.

To simplify notation in the remainder of this proof, we will omit reference to the basis $B$ in the subscripts, writing $r_{\bm w}$ for $r_{\bm w, B}$ and $\T_{\bm w}$ for $\T_{\bm w,B}$.

In what will become the first step in our induction, we will check the coefficient of $\T_{\bm d}$ in both $\T_sh$ and $h\T_s$.  Of course, the coefficient must be the same in each as $h\in Z_\H(\T_s)$.  In $\T_sh$, $\T_{\bm d}$ will appear only in the product $\T_s\T_{\bm{s^{n-1}d}}$, and there with coefficient $r_{\bm{s^{n-1}d}}$.  In $h\T_s$ similarly, $\T_{\bm d}$ will appear only in the product $\T_{\bm{ds^{n-1}}}\T_s$, and there with coefficient $r_{\bm{ds^{n-1}}}$.
Thus, \begin{equation}\label{dn-1}
r_{\bm{ds^{n-1}}}=r_{\bm{s^{n-1}d}}.
\end{equation}

Now suppose for the purposes of induction on $k$ that for $0\le i,j\le k$, we have
\begin{align}
r_{\bm{s^ids^{n-1}}}&=r_{\bm{s^{n-1}ds^i}}\tag{I}\label{hypI}\\
\intertext{and}
r_{\bm{s^ids^j}}&=r_{\bm{s^jds^i}}\tag{II}\label{hypII}
\end{align}
When $k=0$, \eqref{hypI} is the case shown above in \eqref{dn-1}, and (II) is trivial.  We need to show (I) and (II) for when $i$ and $j$ are allowed to be $k+1$.

In the case of (I), we need to show $r_{\bm{s^{k+1}ds^{n-1}}}=r_{\bm{s^{n-1}ds^{k+1}}}$. Consider the coefficient of $\T_{\bm{s^{k+1}d}}$ in $\T_sh$ and $h\T_s$.  In $\T_sh$, $\T_{\bm{s^{k+1}d}}$ appears in the products $\T_s\T_{\bm{s^kd}}$ and $\T_s\T_{\bm{s^{n-1}d}}$, giving the coefficient $r_{\bm{s^kd}}+\xi_{k+1}r_{\bm{s^{n-1}d}}$.
In $h\T_s$, it appears only in $\T_{\bm{s^{k+1}ds^{n-1}}}\T_s$, giving the coefficient $r_{\bm{s^{k+1}ds^{n-1}}}$.  Equating the two gives
\begin{subequations}
\begin{align}\label{k+1d}
r_{\bm{s^{k+1}ds^{n-1}}}&=r_{\bm{s^kd}}+\xi_{k+1}r_{\bm{s^{n-1}d}}.\\
\intertext{On the other hand, if we consider the coefficient of $\T_{\bm{ds^{k+1}}}$ in both $\T_sh$ and $h\T_s$, we obtain the symmetrical equation:}
r_{\bm{s^{n-1}ds^{k+1}}}&=r_{\bm{ds^k}}+\xi_{k+1}r_{\bm{ds^{n-1}}}.
\end{align}
\end{subequations}
Now $r_{\bm{ds^{n-1}}}=r_{\bm{s^{n-1}d}}$ by \eqref{dn-1}, and $r_{\bm{ds^k}}=r_{\bm{s^kd}}$ by induction hypothesis (II), so we have
\begin{equation}\label{k+1dn-1}
r_{\bm{s^{k+1}ds^{n-1}}}=r_{\bm{s^{n-1}ds^{k+1}}},
\end{equation}
which proves (I) in the $k+1$ case.

In the case of (II), we need to show that if one of $i$ or $j$ are set to be $k+1$, we still have $r_{\bm{s^ids^j}}=r_{\bm{s^jds^i}}$ (if both are $k+1$ the statement is trivial).  That is, we need to show $r_{\bm{s^ids^{k+1}}}=r_{\bm{s^{k+1}ds^i}}$ where $i\le k$.

Consider the coefficient of $\T_{\bm{s^{i+1}{d}s^{k+1}}}$ in $\T_sh$ and $h\T_s$.  In $\T_sh$, it appears in the products $\T_s\T_{\bm{s^ids^{k+1}}}$ and $\T_s\T_{\bm{s^{n-1}ds^{k+1}}}$, and in $h\T_s$ it appears in $\T_{\bm{s^{i+1}ds^k}}\T_s$ and $\T_{\bm{s^{i+1}ds^{n-1}}}\T_s$, resulting in the equation:
\begin{subequations}
\begin{align}
r_{\bm{s^ids^{k+1}}}+\xi_{i+1}r_{\bm{s^{n-1}ds^{k+1}}}&=r_{\bm{s^{i+1}ds^k}}+\xi_{k+1}r_{\bm{s^{i+1}ds^{n-1}}}.\label{eq:i+1dk+1coeff}\\
\intertext{Now if $i=k$ we have $\xi_{i+1}r_{\bm{s^{n-1}ds^{k+1}}}=\xi_{k+1}r_{\bm{s^{i+1}ds^{n-1}}}$ by (I), and \eqref{eq:i+1dk+1coeff} then yields $r_{\bm{s^kds^{k+1}}}=r_{\bm{s^{k+1}ds^k}}$.  Thus we may suppose $i<k$.  Now consider the coefficient of $\T_{\bm{s^{i+1}ds^{k+1}}}$ in $\T_sh$ and $h\T_s$.  We obtain a relation symmetric to the above:}
r_{\bm{s^{k+1}ds^i}}+\xi_{i+1}r_{\bm{s^{k+1}ds^{n-1}}}&=r_{\bm{s^kds^{i+1}}}+\xi_{k+1}r_{\bm{s^{n-1}ds^{i+1}}}.
\end{align}
\end{subequations}
From (I) we have $r_{\bm{s^{n-1}ds^{k+1}}}=r_{\bm{s^{k+1}ds^{n-1}}}$ and $r_{\bm{s^{i+1}ds^{n-1}}}=r_{\bm{s^{n-1}ds^{i+1}}}$, and from our induction hypothesis (II) we have $r_{\bm{s^{i+1}ds^k}}=r_{\bm{s^kds^{i+1}}}$ (remembering we have been able to assume $i<k$).  Thus we can conclude
\begin{equation}
r_{\bm{s^ids^{k+1}}}=r_{\bm{s^{k+1}ds^i}}\notag
\end{equation}
and therefore by induction that for $0\le i,j\le n-1$ we have $r_{\bm{s^ids^j}}=r_{\bm{s^jds^i}}$.
\end{proof}

The following theorem allows one to determine exactly the coefficients of terms in an element of the centralizer of $\T_s$ when the relevant double-coset is additive (see Definition \ref{def:additive}), and when the algebra has an $s$-stable basis.

\begin{thm}\label{characterized}
    Let $\H$ be the Hecke algebra of the complex reflection group $W$ considered over $R$.
    Suppose that $\H$ is free as an $R$-module with reduced basis $B=\{\T_{\bm{w},B}\mid w\in W\}$, and consider the double coset $\l s\r d\l s\r$, for $s$ one of the standard generators of $W$ with order $m_s=n$.  For simplicity and without ambiguity, write $\xi_i$ for $\xi_{s,i}$.

    If $\l s\r d\l s\r$ is additive with $d$ the unique distinguished element and $B$ satisfies $\T_{s^i}\T_{\bm d,B}\T_{s^j}=\T_{\bm{s^ids^j},B}$ for any $1\le i,j\le m_s-1$, then the sum $\displaystyle\sum_{0\le i,j<n}r_{\bm{s^ids^j},B}\T_{\bm{s^ids^j},B}$
    with $r_{\bm{s^ids^j},B}\in R$ is in $Z_\H(\T_s)$ if and only if the coefficients satisfy the following:
        \begin{equation}\label{symmetry2}r_{\bm{s^ids^j},B}=r_{\bm{s^jds^i},B}\quad\text{with $0\le i,j\le n-1$}\end{equation}
    and for $1\le i,j\le n-1$,
        \begin{equation}\label{coeffrels}
        r_{\bm{s^ids^j},B}=
         \begin{cases}
\displaystyle r_{\bm{ds^{i+j}},B}+\sum_{k=0}^{i-1}\left(\xi_{i-k}r_{\bm{ds^{j+k}},B}-\xi_{j+k+1}r_{\bm{ds^{i-k-1}},B}\right)&i+j<n,\\ \\
\displaystyle \begin{split}r_{\bm{ds^{i+j-n}},B}+&\xi_{i+j-n+1}r_{\bm{ds^{n-1}},B}+ \\ &\sum_{k=0}^{n-j-2}\left(\xi_{i-k}r_{\bm{ds^{j+k}},B}-\xi_{j+k+1}r_{\bm{ds^{i-k-1}},B}\right)\end{split}
&\text{otherwise.}
         \end{cases}
        \end{equation}
\end{thm}

\begin{proof}
We begin by supposing that $h\in Z_\H(\T_s)$, and show that the given relations hold.  The symmetry statement  \eqref{symmetry2} has already been shown in Lemma~\ref{symmetry}. As with the proof of Lemma~\ref{symmetry}, we drop reference to the basis $B$ in the subscripts for the sake of readability.

Firstly, we deal with the case where $i$ or $j$ is $n-1$.  By Lemma \ref{symmetry} it suffices to prove the case for just one, say $i=n-1$. Consider the coefficients of $\T_{\bm{ds^j}}$ in $\T_sh$ and in $h\T_s$.  They must be equal since $h$ commutes with $\T_s$.  The only occurrence of $\T_{\bm{ds^j}}$ in $\T_sh$ is in the product $\T_s\T_{\bm{s^{n-1}ds^j}}$, and has coefficient $r_{\bm{s^{n-1}ds^j}}$.  The only occurrences of $\T_{\bm{ds^j}}$ in $h\T_s$ are in the products $\T_{\bm{ds^{j-1}}}\T_s$ and $\T_{\bm{ds^{n-1}}}\T_s$, and so it has coefficient $r_{\bm{ds^{j-1}}}+\xi_jr_{\bm{ds^{n-1}}}$.  Thus,
\begin{equation}\label{n-1dj}
r_{\bm{s^{n-1}ds^j}}=r_{\bm{ds^{j-1}}}+\xi_jr_{\bm{ds^{n-1}}},
\end{equation}
and this gives the case when $i=n-1$.

Now we need to deal with the case where $i$ and $j$ are neither $0$ nor $n-1$.  We will show that $r_{\bm{s^ids^j}}$ may be written as a linear combination of $r_{\bm{s^{i-1}ds^{j+1}}}$, $r_{\bm{s^{n-1}ds^{j+1}}}$ and $r_{\bm{s^ids^{n-1}}}$.  A sequence of such steps (re-writing $r_{\bm{s^{i-1}ds^{j+1}}}$ repeatedly) will complete the result.

Compare the coefficient of $\T_{\bm{s^ids^{j+1}}}$ in $\T_sh$ and $h\T_s$.  In $\T_sh$ it occurs in terms $\T_s\T_{\bm{s^{i-1}ds^{j+1}}}$ and $\T_s\T_{\bm{s^{n-1}ds^{j+1}}}$, giving a coefficient of $r_{\bm{s^{i-1}ds^{j+1}}}+\xi_ir_{\bm{s^{n-1}ds^{j+1}}}$.  In $h\T_s$ it occurs in terms $\T_{\bm{s^i{d}s^j}}\T_s$ and $\T_{\bm{s^i{d}s^{n-1}}}\T_s$, giving a coefficient of $r_{\bm{s^ids^j}}+\xi_{j+1}r_{\bm{s^ids^{n-1}}}$.  Thus we have
\begin{equation*}
r_{\bm{s^{i-1}ds^{j+1}}}+\xi_ir_{\bm{s^{n-1}ds^{j+1}}}=r_{\bm{s^ids^j}}+\xi_{j+1}r_{\bm{s^ids^{n-1}}}.
\end{equation*}
This we can re-write (using \eqref{n-1dj}) to give
\begin{align}\label{idj}
r_{\bm{s^ids^j}}  &=r_{\bm{s^{i-1}ds^{j+1}}}+(\xi_ir_{\bm{s^{n-1}ds^{j+1}}}-\xi_{j+1}r_{\bm{s^ids^{n-1}}})\notag\\
&=r_{\bm{s^{i-1}ds^{j+1}}}+\xi_i(r_{\bm{ds^j}}+\xi_{j+1}r_{\bm{ds^{n-1}}})-\xi_{j+1}(r_{\bm{ds^{i-1}}}+\xi_ir_{\bm{ds^{n-1}}})\notag\\
&=r_{\bm{s^{i-1}ds^{j+1}}}+\xi_ir_{\bm{ds^j}}-\xi_{j+1}r_{\bm{ds^{i-1}}}.
\end{align}

An elementary induction using \eqref{idj} as the base step then shows that for $1\le m\le \min(i,n-j-1)$ we have
\begin{equation}\label{eq:idj.induction}
r_{\bm{s^ids^j}}=r_{\bm{s^{i-m}ds^{j+m}}}+\sum_{k=0}^{m-1}\left(\xi_{i-k}r_{\bm{ds^{j+k}}}-\xi_{j+k+1}r_{\bm{ds^{i-k-1}}}\right).
\end{equation}
We then have two cases depending on whether $i\le n-j-1$ (the case $i+j<n$) or otherwise.  If $i\le n-j-1$ then $i=\min(i,n-j-1)$ so setting $m=i$ in \eqref{eq:idj.induction} gives the first part of \eqref{coeffrels}.  If on the other hand $n-j-1<i$ ($i+j\ge n$) then setting $m=n-j-1$ in \eqref{eq:idj.induction} gives:
\begin{equation}
  r_{\bm{s^ids^j}}=r_{\bm{s^{i+j+1-n}ds^{n-1}}}+\sum_{k=0}^{n-j-2}\left(\xi_{i-k}r_{\bm{ds^{j+k}}}-\xi_{j+k+1}r_{\bm{ds^{i-k-1}}}\right).
\end{equation}
Combining this with the fact that
\begin{equation}
  r_{\bm{s^{i+j+1-n}ds^{n-1}}}=r_{\bm{s^{n-1}ds^{i+j+1-n}}}=r_{\bm{ds^{i+j-n}}}+\xi_{i+j+1-n}r_{\bm{ds^{n-1}}},
\end{equation}
by \eqref{n-1dj}, gives the second case of \eqref{coeffrels} and the forward direction is complete.

It remains to show the reverse implication: that an element whose coefficients satisfy \eqref{symmetry2} and \eqref{coeffrels} must centralize $\T_s$.  To show an element commutes with $\T_s$, we may restrict our attention to those elements from a given $\l s\r$-$\l s\r$ double coset.

We proceed in a very similar vein to the forward direction.  Suppose $h\in\H$ has expression in terms of the basis $B$ with coefficients satisfying the relations given in the Theorem statement.  We will show that the coefficient of any given $\T_{\bm{w}}$ in $h\T_s$ is equal to the coefficient of $\T_{\bm{w}}$ in $\T_sh$.  As before, we will need to look at different cases of $w$, and will need only look at a particular $\l s\r$-$\l s\r$ double coset.  That is, we will take $h$ to be a linear combination of elements of a double coset:
\[h=\sum_{0\le i,j\le n-1}r_{\bm{s^ids^j}}\T_{\bm{s^i\bm{d}s^j}}\]
where $d$ is a minimal length $\l s\r$-$\l s\r$ double coset representative.  Assuming coefficients satisfy \eqref{symmetry2} and \eqref{coeffrels}, and using the symmetry relation \eqref{symmetry2} to collect terms in the first summation below, we then have for $1\le i,j\le n-1$:
\begin{align}
h=&r_{\bm d}\T_{\bm d}+\sum_{i=1}^{n-1}r_{\bm{ds^i}}(\T_{\bm{ds^i}}+\T_{\bm{s^i\bm{d}}})+\nonumber\\
  &\sum_{i+j<n}\Bigl(r_{\bm{ds^{i+j}}}+\sum_{k=0}^{i-1}\left(\xi_{i-k}r_{\bm{ds^{j+k}}}-\xi_{j+k+1}r_{\bm{ds^{i-k-1}}}\right)\Bigr)\T_{\bm{s^i\bm{d}s^j}}+\nonumber\\
  &\sum_{i+j\ge n}\left(r_{\bm{ds^{i+j-n}}}+\xi_{i+j-n+1}r_{\bm{ds^{n-1}}}+ \sum_{k=0}^{n-j-2}\left(\xi_{i-k}r_{\bm{ds^{j+k}}}-\xi_{j+k+1}r_{\bm{ds^{i-k-1}}}\right)\right) \T_{\bm{s^i\bm{d}s^j}}.\label{genericdcosetsum}
\end{align}

We start as before by looking at the coefficient of $\T_{\bm{ds^j}}$ ($j>0$) in both $h\T_s$ and $\T_sh$.

In $\T_sh$ we will have such terms arising in products $\T_s\T_{\bm{s^{n-1}\bm{d}s^j}}$ with coefficient $\xi_0=1$, so its coefficient will be $r_{\bm{s^{n-1}ds^{j}}}=r_{\bm{ds^{j-1}}}+\xi_jr_{\bm{ds^{n-1}}}$ (since the coefficients in $h$ satisfy equations~\eqref{symmetry2} and~\eqref{coeffrels}).  In $h\T_s$ such terms will arise in the products $\T_{\bm{ds^{j-1}}}\T_s$ with coefficient 1 and in $\T_{\bm{ds^{n-1}}}\T_s$ with coefficient $\xi_j$.  Thus its coefficient in $h\T_s$ will be $1(r_{\bm{ds^{j-1}}})+\xi_j(r_{\bm{ds^{n-1}}})$ as it is in $\T_sh$.

Now we must show that the coefficients of terms of form $\T_{\bm{s^i\bm{d}s^j}}$ for $i,j\ge 1$ are the same in both $h\T_s$ and $\T_sh$.

Begin by supposing $i+j\le n$, and consider the coefficient of the term $\T_{\bm{s^i\bm{d}s^j}}$ in $h\T_s$.  This term will arise through the products $\T_{\bm{s^i\bm{d}s^{j-1}}}\T_s$ with coefficient 1, and $\T_{\bm{s^i\bm{d}s^{n-1}}}\T_s$ with coefficient $\xi_j$.  We have $i+(j-1)< n$ and $i+(n-1)\ge n$ and so our coefficient of $\T_{\bm{s^i\bm{d}s^j}}$ in $h\T_s$ is:
\begin{align*}
1&\left(r_{\bm{ds^{i+j-1}}}+\sum_{k=0}^{i-1}\left(\xi_{i-k}r_{\bm{ds^{j-1+k}}}-\xi_{j+k}r_{\bm{ds^{i-k-1}}}\right)\right)+ \xi_j(r_{\bm{ds^{i-1}}}+\xi_{i}r_{\bm{ds^{n-1}}})\\
 &=r_{\bm{ds^{i+j-1}}}+ \xi_i\left(r_{\bm{ds^{j-1}}}+\xi_jr_{\bm{ds^{n-1}}}\right)+
 \sum_{k=1}^{i-1}\left(\xi_{i-k}r_{\bm{ds^{j-1+k}}}-\xi_{j+k}r_{\bm{ds^{i-k-1}}}\right).
\end{align*}

Correspondingly, the coefficient of $\T_{\bm{s^i\bm{d}s^j}}$ in $\T_sh$ will arise in the products $\T_s\T_{\bm{s^{i-1}\bm{d}s^j}}$ with coefficient 1, and $\T_s\T_{\bm{s^{n-1}\bm{d}s^j}}$ with coefficient $\xi_i$.  This gives a coefficient of
\[r_{\bm{ds^{i-1+j}}}+\sum_{k=0}^{i-2}\left(\xi_{i-k-1}r_{\bm{ds^{j+k}}}-\xi_{j+k+1}r_{\bm{ds^{i-k-2}}}\right)+\xi_i\left(r_{\bm{ds^{j-1}}}+\xi_jr_{\bm{ds^{n-1}}}\right)
\]
(noting that when $i=n-1$, $\sum_{k=0}^{n-j-2}\left(\xi_{i-k}r_{\bm{ds^{j+k}}}-\xi_{j+k+1}r_{\bm{ds^{i-k-1}}}\right)=0$). This simplifies to the above coefficient in $h\T_s$.

It remains to show that the coefficients of $\T_{\bm{s^i\bm{d}s^j}}$ in $\T_sh$ and $h\T_s$ are equal when $i+j>n$.
Again this term arises in the products $\T_{\bm{s^i\bm{d}s^{j-1}}}\T_s$ with coefficient 1, and $\T_{\bm{s^i\bm{d}s^{n-1}}}\T_s$ with coefficient $\xi_j$.  This time, $i+(j-1)$ and $i+(n-1)$ are both $\ge n$, and so our coefficient of $\T_{\bm{s^i\bm{d}s^j}}$ in $h\T_s$ is:
\begin{multline}\label{eq:hTs.coeff}
r_{\bm{ds^{i+j-1-n}}}+\xi_{i+j-n}r_{\bm{ds^{n-1}}}+ \sum_{k=0}^{n-j-1}\left(\xi_{i-k}r_{\bm{ds^{j+k-1}}}-\xi_{j+k}r_{\bm{ds^{i-k-1}}}\right)+ \\ \xi_j\left(r_{\bm{ds^{i-1}}}+\xi_{i}r_{\bm{ds^{n-1}}}\right).
\end{multline}
Similarly, the coefficient in $\T_sh$ is from the term's appearance in the product $\T_s\T_{\bm{s^{i-1}\bm{d}s^j}}$ with coefficient 1 and $\T_s\T_{\bm{s^{n-1}\bm{d}s^j}}$ with coefficient $\xi_i$.  Since the coefficient in $h$ satisfy the symmetry relation~\eqref{symmetry2}, $r_{\bm{s^{i-1}{d}s^j}}=r_{\bm{s^{j}{d}s^{i-1}}}$ and $r_{\bm{s^{n-1}{d}s^j}}=r_{\bm{s^j{d}s^{n-1}}}$.  Thus the coefficient of $\T_{\bm{s^i\bm{d}s^j}}$ in $\T_sh$ is
\begin{multline}\label{eq:Tsh.coeff}
r_{\bm{ds^{i+j-1-n}}}+\xi_{i+j-n}r_{\bm{ds^{n-1}}}+ \sum_{k=0}^{n-i-1}\left(\xi_{j-k}r_{\bm{ds^{i+k-1}}}-\xi_{i+k}r_{\bm{ds^{j-k-1}}}\right)+ \\ \xi_i\left(r_{\bm{ds^{j-1}}}+\xi_{j}r_{\bm{ds^{n-1}}}\right).
\end{multline}
A careful check of these summations shows that the expressions~\eqref{eq:hTs.coeff} and~\eqref{eq:Tsh.coeff} are equal.
It follows that $h\T_s=\T_sh$ and the proof is complete.
\end{proof}

\begin{rem}
    Theorem \ref{characterized} is a generalization of \cite[Lemma (2.4)]{DJ87} and \cite[Lemma (3.1)]{Fmb} to generators of order greater than or equal to $2$.
\end{rem}
    Theorem \ref{characterized} has as a consequence an $s$-conjugacy class analogue of results (due to Ram \cite{Ram91} in type $A$ and Geck and Pfeiffer \cite{GP93} for general finite Coxeter groups) which state that coefficients of terms in central elements of a Hecke algebra can be written as linear combinations of coefficients of shortest elements of conjugacy classes.  Note that all the coefficients on the right hand side of equations \eqref{coeffrels} are shortest in their $s$-conjugacy classes.

    However there is an important distinction between generators of order two and higher here. In a double coset of a generator of order two the shortest elements all either appear on the right hand side of the equations in the Theorem statement, or they are accounted for by Lemma \ref{symmetry}.  This forces all coefficients of minimal length elements to be equal.  When the generator has higher order, not all minimal elements are accounted for in this way, and so not all corresponding coefficients are equal.  For instance, if $s^3=1$ and $\l s\r d\l s\r$ is additive, then $\{ds^2, s^2d, sds\}$ are all shortest elements of the same $s$-conjugacy class but $r_{sds}$ is dependent on other coefficients: $r_{sds}=r_{ds^2}+\xi_1r_{ds}-\xi_2r_d$ (see equation~\eqref{coeffrels} in Theorem~\ref{characterized}).

Note that the coefficients of the $r_w$ in the right hand sides of the expressions for $r_{s^ids^j}$ in Theorem \ref{characterized} are $1$ if and only if $w$ is $s$-conjugate to $s^ids^j$.

\begin{cor}
Let $\H$ be the Hecke algebra of a finite complex reflection group $W$ with reduced basis $B$, and $s\in S$.

If $\l s\r d\l s\r$ is additive and $B$ satisfies $\T_{s^i}\T_{\bm d,B}\T_{s^j}=\T_{\bm{s^ids^j},B}$ for any $1\le i,j\le m_s-1$
(or centralizing with $\T_{s^i}\T_{\bm d,B}=\T_{\bm{s^id},B}$ for any $1\le i\le m_s-1$), then the coefficient $r_{\bm w,B}$ for $w\in\l s\r d\l s\r$ of any $\T_{\bm w,B}$ appearing in an element $h$ of $Z_\H(\T_s)$ may be written as an $ R$-linear combination of the coefficients of shortest elements of $s$-conjugacy classes in $h$.
\end{cor}

\begin{proof}
The centralizing case is trivial as in this case each $w\in\l s\r d\l s\r$ is itself shortest in its (singleton) conjugacy class.  The additive case is immediate from the Theorem, by inspection of the right hand side of equations \eqref{coeffrels}.
\end{proof}

\subsection{Basis elements}\label{sec:basis.thm.for.centralizers}

We collect here sufficient conditions for a set to be an integral basis of a centralizer of a parabolic subalgebra.
Suppose $\H$ is the Hecke algebra of a complex reflection group $W$, with standard generating set $S$.

\begin{prop}\label{basis}
Let $W$ be a finite complex reflection group with $r$ $J$-conjugacy classes (for $J\subseteq S$), whose Hecke algebra $\H$ is a free $R$-module with reduced basis $B=\{\T_{\bm w,B}\mid w\in W\}$.

We assume the existence of a set of $r$ ``distinguished'' reduced expressions
\[\mathcal M:=\{\bm w_i\in\eta_B(w_i)\mid w_i\in C_i\in ccl_J(W)\}\subset Red(S)\]
indexed by the $J$-conjugacy classes $ccl_J(W)$ of $W$, together with relations
\begin{equation}\mathcal R:=\{r_{\bm g,B}=\sum_{w\in\mathcal M}a_{\bm{g},\bm w}r_{\bm w,B}\mid a_{\bm g,\bm w}\in R\}\label{eq:basis:rels}\end{equation}
with the property that $h=\sum_{g\in W}r_{\bm g,B}\T_{\bm g,B}\in Z_\H(\H_J)$ if and only if the relations $\mathcal R$ hold.

Then there exists a set $\B=\{b_1,\dots,b_r\}\subseteq Z_\H(\H_J)$ satisfying
    \begin{enumerate}
    \item\label{Bcriterion2} Each $b_i$ contains a term $\T_{\bm w_i,B}$ for $\bm{w_i}\in\mathcal M$ with coefficient $1$ which does not appear with non-zero coefficient in any other $b_j$ ($j\neq i$),
    \item $\B$ is an $ R$-basis for $Z_\H(\H_J)$.
    \end{enumerate}
\end{prop}

\begin{defn}
Elements $b_i$ of $\B$ satisfying the properties of the theorem will be called \emph{$J$-class elements}.
\end{defn}

\begin{proof}
The assumption of the existence of the set $\mathcal M$ with the properties as declared means that an element of the centralizer is uniquely determined by the values of its coefficients of elements $\T_{\bm w,B}$ corresponding to the elements $\bm w$ of $\mathcal M$.
It also follows immediately from the assumptions that any element of $\H$ whose coefficients satisfy the relations $r_{\bm g,B}=\sum_{w\in\mathcal M}a_{\bm g,\bm w}r_{\bm w,B}$ is in the centralizer.  Thus setting $r_{\bm{w'},B}=\delta_{\bm w,\bm w'}$ where $\bm w'$ runs over the elements of $\mathcal M$, uniquely defines an element of the centralizer corresponding to the element $\bm w\in\mathcal M$.   We will denote this element $b_i$, corresponding to the element $\bm{w_i}\in\mathcal M$ (which corresponds to the $J$-conjugacy class $C_i$).

We claim that the set $\B=\{b_i\mid C_i\in ccl_J(W)\}$ is the set whose existence is asserted in the theorem.  Claim \eqref{Bcriterion2} follows trivially from the definition of $b_i$.  Note also that the linear independence of the set $\B$ follows immediately from part \eqref{Bcriterion2} of the Claim.  Thus it remains to show that $\B$ spans $Z_\H(\H_J)$ over $ R$.

Suppose $h\in Z_\H(\H_J)$, and write $h$ with respect to $B$, $h=\sum_{g\in W}r_{\bm g,B}\T_{\bm g,B}$ with $r_{\bm g,B}\in R$.  Then we may write $h=\sum_{w_i\in\mathcal M}r_{\bm w_i}\T_{\bm w_i}+\sum_{\T_{\bm g}\in B,g\not\in\mathcal M}r_{\bm g,B}\T_{\bm g,B}$.  We then have that
\[h':=h-\sum_{1\le i\le r}r_{\bm{w_i}}b_i\in Z_\H(\H_J),\]
however this new element $h'$ of the centralizer contains no terms from $\mathcal M$ with non-zero coefficient.  That is, for $\bm{g_i}\in\mathcal M$, the coefficient of $\T_{\bm g_i}$ in $h'$ is zero.  It follows that $h'=0$ as all coefficients of terms $\T_{\bm g,B}$ in $h'$ are linear combinations of the coefficients of the terms corresponding to elements of $\mathcal M$ (which are all zero).  Thus $h=\sum_{1\le i\le r}r_{\bm{g_i}}b_i$, and so is in $span_{ R}(B)$.  Thus $\B$ is an $ R$-basis for $Z_\H(\H_J)$.
\end{proof}

\begin{cor}\label{cor:basis+classes}
If in addition we have the property that in the relations \eqref{eq:basis:rels} we have $a_{\bm g,\bm w}=1$ if and only if $w$ and $g$ are conjugate, and all other $a_{\bm g,\bm{w'}}$ specialize to zero, then the basis elements $\B$ specialize to conjugacy class sums in the group algebra.
\end{cor}

The scenario of Corollary \ref{cor:basis+classes} is, of course, the familiar situation in which we find ourselves for the finite Coxeter groups, but also applies for $G_4$ (see Corollary \ref{cor:basis+classes+g4} below).

The Proposition and its Corollary are an attempt to formalize the tools used to find such integral bases in real reflection groups in a way which allows the generalization to complex reflection groups.  In the real case, the set $\mathcal M$ can be any selection of representatives of minimal elements of the $J$-conjugacy classes. In the complex case it is necessary to select more carefully, and it is not clear firstly whether such a set exists in general and secondly if it does then whether there is a general characterization of such distinguished terms.

\section{Type $G_4$}\label{sec:g4}

In this section we focus our attention on the Hecke algebra of the complex reflection group $G_4$.  In Section \ref{sec:g4.centralizers}, after introducing some properties of the group and its associated Hecke algebra (Sections \ref{sec:g4.hecke.algebra} and \ref{sec:g4.hecke.bases}) we use the results of Section \ref{sec:additive.d.cosets.coeff.rels} to describe the centralizers of the generators.  In Section \ref{sec:g4.centre} we use these centralizer bases to construct a basis for the centre.  The principal difficulty in combining the centralizer bases is that they are each with respect to different bases for the Hecke algebra.  In order to combine them into a basis for the centre they need re-writing with respect to a common basis for $\H$, which we do in Section \ref{sec:g4.hecke.bases.rewriting}.

\subsection{The Hecke algebra of type $G_4$}\label{sec:g4.hecke.algebra}

The complex reflection group $G_4$ is generated by two pseudo-reflections $S=\{s,t\}$, each of order $3$, and who together satisfy the standard braid relation.  That is,
\[G_4=\l s,t\mid s^3=t^3=1,\ sts=tst\r.\]
The group can be represented by a Coxeter-type diagram as in \cite{BMR98} (see Figure \ref{G4diagram}).

\begin{figure}[h]
   \begin{minipage}[t]{\linewidth}
        \centering
        \includegraphics{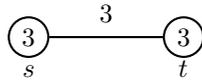}
    \caption{Coxeter-type diagram of the group $G_4$}\label{G4diagram}
   \end{minipage}%
\end{figure}

The group has $24$ elements and seven conjugacy classes:
\[\{1\}, \{s,t,s^2ts,sts^2\}, \{s^2,t^2,s^2t^2s,st^2s^2\}, \{st,ts,t^2st^2,s^2ts^2\},\]
\[\{s^2t^2,t^2s^2,st^2s,ts^2t\}, \{ts^2,s^2t,st^2,t^2s,sts,s^2t^2s^2\}, \{ststst\}.\]

There are several useful relations among words in $s$ and $t$:
\[s^it^js^k=t^ks^jt^i\quad\text{for }i,j,k\in\{1,2\},\ j=i\text{ and/or }k.\]

The centralisers of the generators are $Z_{G_4}(s)=\{1,s,s^2,ts^2t,t^2st^2,tststs\}$ and $Z_{G_4}(t)=\{1,t,t^2,st^2s,s^2ts^2,tststs\}$, and the centre is $Z(G_4)=\{1,ststst\}$.  The $\l s\r$-$\l s\r$ double coset representatives are $\{1,ts^2t,t,t^2\}$ and the $\l t\r$-$\l t\r$ double coset representatives are $\{1,st^2s,s,s^2\}$.

Note that all $\l s\r$-$\l s\r$ and $\l t\r$-$\l t\r$ double cosets in $G_4$ either centralize $s$ (resp. $t$) or are additive.

The Hecke algebra of $G_4$ is the algebra generated over $R=\Z[\xi_1,\xi_2]$ by $\{\T_s,\T_t\}$ subject to the analogue of the braid relation and the order relation
\begin{equation}\label{xiHpower}
\T_x^3=1+\xi_{1}\T_x+\xi_{2}\T_{x^2}
\end{equation}
for $x\in S$.  Note that since $s$ and $t$ are conjugate, $\xi_{s,i}=\xi_{t,i}=:\xi_i$.

\subsection{Bases for the Hecke algebra}\label{sec:g4.hecke.bases}
As noted in Section \ref{sec:heckealg}, different reduced expressions for the same group element can yield different elements of the algebra $\H$.  In what follows we will use two bases for $\H$, one appropriate to each of the centralizers $Z_\H(\T_s)$ and $Z_\H(\T_t)$. When we proceed to find an integral basis for the centre of $\H$ we will have to choose one of these.  Thus, in order to study the centre we will need to be able to transform between the bases.

Because for much of this section we will be representing basis elements by explicit reduced expressions, we will drop the reference to the basis $B$ from the notation where appropriate, writing for instance $\T_{s^2ts}$ to represent $\T_{s^2}\T_t\T_s$ instead of the more general $\T_{\bm{s^2ts},B}$, which refers to the reduced expression in the basis $B$ that corresponds to the group element $s^2ts$.

The bases we will use are as follows.

\noindent
The $sts$ basis:
\begin{multline*}
B_{sts}=\{\T_1, \T_s, \T_t, \T_{s^2}, \T_{t^2}, \T_{st}, \T_{ts}, \T_{st^2}, \T_{t^2s}, \T_{s^2t}, \T_{ts^2}, \T_{s^2t^2}, \T_{t^2s^2}, \T_{st^2s}, \\
\T_{ts^2t}, \T_{s^2ts^2}, \T_{t^2st^2}, \T_{ststst}, \T_{sts}, \T_{sts^2}, \T_{s^2ts}, \T_{s^2t^2s}, \T_{st^2s^2}, \T_{s^2t^2s^2}\},
\end{multline*}
and the $tst$ basis:
\begin{multline*}
B_{tst}=\{\T_1, \T_s, \T_t, \T_{s^2}, \T_{t^2}, \T_{st}, \T_{ts}, \T_{st^2}, \T_{t^2s}, \T_{s^2t}, \T_{ts^2}, \T_{s^2t^2}, \T_{t^2s^2}, \T_{st^2s}, \\
\T_{ts^2t}, \T_{s^2ts^2}, \T_{t^2st^2}, \T_{ststst}, \T_{tst}, \T_{t^2st}, \T_{tst^2}, \T_{ts^2t^2}, \T_{t^2s^2t}, \T_{t^2s^2t^2}\}.
\end{multline*}

Note that all but the last six elements of these bases are identical.  The last six elements of each basis are listed so that equal elements in $G_4$ correspond, but with different reduced expressions. For these we will describe the transformations between bases.

Firstly, if $\bm w$ and $\bm{w'}$ are two reduced expressions for the same group element and $\bm{w}$ can be reached from $\bm{w'}$ by repeated application of the braid relation alone, then as mentioned above $\T_{\bm w}=\T_{\bm{w'}}$ (Subsection~\ref{subsubsec:R-mod}).  This is the case for the first three of the remaining terms ($T_{sts}=T_{tst}$, $T_{sts^2}=T_{t^2st}$, $T_{s^2ts}=T_{tst^2}$).  For example, \[\T_{s^2ts}=\T_s\T_s\T_t\T_s=\T_s\T_t\T_s\T_t=\T_t\T_s\T_t\T_t=\T_{tst^2}.\]
The remaining relations, $s^2t^2s=ts^2t^2$, $st^2s^2=t^2s^2t$, and $s^2t^2s^2=t^2s^2t^2$ are not so neatly transcribed to the Hecke algebra, as the transformation needs the substitution of $1=\T_s^3-\xi_2\T_{s^2}-\xi_1\T_s=\T_t^3-\xi_2\T_{t^2}-\xi_1\T_t$.
For brevity, we omit the details, but the following relations can easily be shown to hold:
\begin{align}\label{t2s2t}
\T_{s^2t^2s}-\T_{ts^2t^2}&=\xi_2(\T_{st^2s}-\T_{ts^2t})+\xi_1(\T_{t^2s}-\T_{ts^2}),\\
\T_{st^2s^2}-\T_{t^2s^2t}&=\xi_2(\T_{st^2s}-\T_{ts^2t})+\xi_1(\T_{st^2}-\T_{s^2t}),\text{ and}\notag\\
\T_{s^2t^2s^2}-\T_{t^2s^2t^2}&=\xi_2(\T_{s^2ts^2}-\T_{t^2st^2})+\xi_2^2(\T_{st^2s}-\T_{ts^2t})+\notag\\
 &\hspace{2cm}\xi_1\xi_2(\T_{t^2s}-\T_{s^2t}+\T_{st^2}-\T_{ts^2})+\xi_1(\T_s-\T_t).\notag
\end{align}

\begin{rem}
Note that these relations all agree with the ordering in Proposition~\ref{prop:bremal97}, which relates specifically to the Hecke algebras of $G(r,1,n)$.  It would be interesting to know whether this were true for Hecke algebras of complex reflection groups wherever a reduced basis exists.
\end{rem}

Recall the definition of an $x$-stable basis, for $x\in S$ (Definition \ref{def:stablebasis}).  For the Hecke algebra of $G_4$, there is an $x$-stable basis for each $x\in S$ as follows (the proof involves elementary checking and is omitted).

\begin{lem}\label{lem:sgoodbasis:G4}
The Hecke algebra $\H$ of $G_4$ has $x$-stable bases for each $x\in S$.  $B_{sts}$ is an $s$-stable basis for $\H$ and $B_{tst}$ is a $t$-stable basis for $\H$.
\end{lem}

This implies that all $\l s\r$-$\l s\r$ and $\l t\r$-$\l t\r$ double cosets in $G_4$ are stable with respect to $B_{sts}$ and $B_{tst}$ respectively, and hence that the corresponding double coset graphs are all trivial.


\subsection{Centralizers in the Hecke algebra of type $G_4$}\label{sec:g4.centralizers}

The explicit consequences of Corollary \ref{cor:basis+classes} for the group $G_4$ are the following, noting that the Hecke algebra $\H$ of $G_4$ has an $s$-stable basis (Lemma \ref{lem:sgoodbasis:G4}), and that every double coset $\l s\r d\l s\r$ in $G_4$ is either in the centralizer $Z_W(s)$ or is additive.

\begin{cor}\label{cor:coeffrelsg4} Let $\H$ be the Hecke algebra of $G_4$, with respect to the basis $B_{sts}$.  Then $h=\sum_{w\in G_4}r_{\bm w,B_{sts}}\T_{\bm w,B_{sts}}\in Z_\H(\T_s)$ if and only if the following relations hold:
  \[
  \begin{array}[t]{rr@{\,=\,}l}
  \textsc{\tiny [S1]}&r_{s^2ts}=r_{sts^2}     &r_t+\xi_1r_{ts^2}\cr
  \textsc{\tiny [S2]}&r_{s^2t^2s}=r_{st^2s^2} &r_{t^2}+\xi_1r_{t^2s^2}\cr
  \textsc{\tiny [S3]}&r_{sts}                   &r_{ts^2}+\xi_1r_{ts}-\xi_2r_{t}\cr
  \textsc{\tiny [S4]}&r_{st^2s}                 &r_{t^2s^2}+\xi_1r_{t^2s}-\xi_2r_{t^2}\cr
  \textsc{\tiny [S5]}&r_{s^2ts^2}             &r_{ts}+\xi_2r_{ts^2}\cr
  \textsc{\tiny [S6]}&r_{s^2t^2s^2}             &r_{t^2s}+\xi_2r_{t^2s^2}\cr
  \textsc{\tiny [S7]}&r_{st}                &r_{ts}\cr
  \textsc{\tiny [S8]}&r_{st^2}                &r_{t^2s}\cr
  \textsc{\tiny [S9]}&r_{s^2t}                &r_{ts^2}\cr
  \textsc{\tiny [S10]}&r_{s^2t^2}                &r_{t^2s^2}\cr
  \end{array}
  \]
\end{cor}

\newcommand{\rel}[1]{\textsc{\tiny [#1]}}

\begin{proof}
  These relations follow from Theorem \ref{characterized} because both of the non-centralizing $s$-$s$ double cosets are additive ($d=t$ or $t^2$).  The first six relations $\rel{S1}$ to $\rel{S6}$ follow from these two centralizing double cosets and equations \eqref{coeffrels}.  The latter four $\rel{S7}$ - $\rel{S10}$ (and parts of the first two) are the symmetric relations \eqref{symmetry2}.
\end{proof}

This tells us that to find elements of the centralizer of $\T_s$ in $\H$ all we need do is set values for the coefficients of form $r_{ds^j}$ for each $j=0,1$, or $2$ and $d=t$ or $t^2$.  The elements of these forms --- $d$, $ds$ and $ds^2$ --- are shortest elements of $s$-conjugacy classes.

\begin{cor}
\label{cor:basis+classes+g4}
The centralizer $Z_\H(\T_s)$ of a generator $\T_s$ of the Hecke algebra $\H$ of $G_4$ has an integral basis indexed by the $s$-conjugacy classes whose elements specialize to $s$-conjugacy class sums.
\end{cor}
\begin{proof}
The relations given in Corollary \ref{cor:coeffrelsg4} satisfy the requirements of Proposition \ref{basis} and of Corollary \ref{cor:basis+classes}, with distinguished set \[\mathcal M=\{1,s,s^2,ts^2t,ts^2ts,tststs\}\cup\{t,ts,ts^2,t^2,t^2s,t^2s^2\}.\]  The Corollary follows.
\end{proof}

We have the following basis for the centralizer of $\T_s$ in $\H$.

\begin{prop}\label{s-elts}
Let $\H$ be the Hecke algebra of the group $G_4$, with basis $B_{sts}$.  The following elements together form an $R$-basis for $Z_\H(\T_s)$:

\begin{tabular}{llllll}
(1) $\T_1$,         &(2) $\T_{s}$,     & (3) $\T_{s^2}$,
& (4) $\T_{ts^2t}$,   &(5) $\T_{ts^2ts}$,& (6) $\T_{tststs}$,     \\
\multicolumn{3}{l}{(7) $\T_t+\T_{s^2ts}+\T_{sts^2}-\xi_2\T_{sts}$,}&
\multicolumn{3}{l}{(8) $\T_{t^2}+\T_{s^2t^2s}+\T_{st^2s^2}-\xi_2\T_{st^2s}$,}\\
\multicolumn{3}{l}{(9)  $\T_{st}+\T_{ts}+\T_{s^2ts^2}+\xi_1\T_{sts}$,}&
\multicolumn{3}{l}{(10) $\T_{st^2}+\T_{t^2s}+\T_{s^2t^2s^2}+\xi_1\T_{st^2s}$,}\\
\multicolumn{6}{l}{(11)     $\T_{s^2t}+\T_{ts^2}+\T_{sts}+\xi_1\T_{s^2ts}+\xi_1\T_{sts^2}+\xi_2\T_{s^2ts^2}$,}\\
\multicolumn{6}{l}{(12) $\T_{s^2t^2}+\T_{t^2s^2}+\T_{st^2s}+\xi_1\T_{s^2t^2s}+\xi_1\T_{st^2s^2}+\xi_2\T_{s^2t^2s^2}$.}
\end{tabular}

\end{prop}

\begin{proof} That elements (1) to (6) commute with $\T_s$ is easily checked.  In fact, the first three are entirely trivial and the second three follow from checking the fourth: \[\T_{ts^2t}\T_s=\T_t\T_s\T_s\T_t\T_s=\T_t\T_s\T_t\T_s\T_t=\T_s\T_t\T_s\T_s\T_t=\T_s\T_{ts^2t}.\]

The remaining elements (7) to (12) all correspond to double cosets which are additive (with distinguished elements $t$ or $t^2$), which means that Corollary \ref{cor:coeffrelsg4} applies.  The coefficients in these elements all satisfy the relations of Corollary \ref{cor:coeffrelsg4}, so are all in $Z_\H(\T_s)$.  They also each specialize to an $s$-conjugacy class sum under $\rho$ (recalling that $\rho(\xi_1)=\rho(\xi_2)=0$), and together with the first six they correspond to a complete set of the twelve $s$-conjugacy classes in $G_4$.

Finally, each contains at least one  ``distinguished" term which does not appear in any other element.  Thus the set satisfies the requirements of Proposition \ref{basis} and so is an $R$-basis for $Z_\H(\T_s)$.
\end{proof}

Note that the basis for $Z_\H(\T_s)$ given above is in terms of the basis $B_{sts}$ for the Hecke algebra $\H$.  Exchanging each $s$ in the theorem with $t$ and vice versa we obtain a basis for the centralizer of $\T_t$, however this basis will now be in terms of the basis $B_{tst}$ for $\H$:

\begin{cor}
Let $\H$ be the Hecke algebra of $G_4$, with respect to $B_{tst}$.  The following elements together form an $ R$-basis for $Z_\H(\T_t)$:
\noindent
\begin{tabular}{llllll}
(1) $\T_1$,         & (2) $\T_{t}$,          & (3) $\T_{t^2}$,&
(4) $\T_{st^2s}$,   & (5) $\T_{s^2ts^2}$,    & (6) $\T_{ststst}$.\\
\multicolumn{3}{l}{(7) $\T_s+\T_{t^2st}+\T_{tst^2}-\xi_2\T_{tst}$,}&
\multicolumn{3}{l}{(8) $\T_{s^2}+\T_{t^2s^2t}+\T_{ts^2t^2}-\xi_2\T_{ts^2t}$,}\\
\multicolumn{3}{l}{(9) $\T_{ts}+\T_{st}+\T_{t^2st^2}+\xi_1\T_{tst}$,}&
\multicolumn{3}{l}{(10) $\T_{ts^2}+\T_{s^2t}+\T_{t^2s^2t^2}+\xi_1\T_{ts^2t}$,}\\
\multicolumn{6}{l}{(11) $\T_{t^2s}+\T_{st^2}+\T_{tst}+\xi_1\T_{t^2st}+\xi_1\T_{tst^2}+\xi_2\T_{t^2st^2}$,}\\
\multicolumn{6}{l}{(12) $\T_{t^2s^2}+\T_{s^2t^2}+\T_{ts^2t}+\xi_1\T_{t^2s^2t}+\xi_1\T_{ts^2t^2}+\xi_2\T_{t^2s^2t^2}$.}
\end{tabular}
\end{cor}

\subsubsection{Re-writing bases for centralizers}\label{sec:g4.hecke.bases.rewriting}
As noted previously, we would like to combine the information we have about the centralizers $Z_\H(\T_s)$ and $Z_\H(\T_t)$ to obtain a basis for the centre, and for this reason we need to write both bases with respect to the same basis for $\H$.  Making an arbitrary (though alphabetically pleasing) choice, we will write both in terms of $B_{sts}$, meaning we now need to transform the basis for $Z_\H(\T_t)$ (given with respect to $B_{tst}$) into the basis $B_{sts}$, and we will use the relations found in Section \ref{sec:g4.centralizers}.

We omit the computations and summarize the results as follows:

\begin{prop}\label{t-elts}
Let $\H$ be the Hecke algebra of $G_4$ with basis $B_{sts}$.  An $R$-basis for $Z_\H(\T_t)$ is the set consisting of:

\noindent
\begin{tabular}{rllllll}
(1)&$\T_1$, &(2) $\T_{st^2s}$,&(3) $\T_t$,&(4) $\T_{s^2ts^2}$,&(5) $\T_{t^2}$,&(6) $\T_{ststst}$,\\
(7)&\multicolumn{3}{l}{ $\T_s+\T_{sts^2}+\T_{s^2ts}-\xi_2\T_{sts}$,}
&\multicolumn{3}{l}{(8) $\T_{ts}+\T_{st}+\T_{t^2st^2}+\xi_1\T_{sts}$,}\\
(9)& \multicolumn{6}{l}{$\T_{t^2s}+\T_{st^2}+\T_{sts}+\xi_1\T_{sts^2}+\xi_1\T_{s^2ts}+\xi_2\T_{t^2st^2}$,}\\
(10)&\multicolumn{6}{p{.85\textwidth}}{ $\T_{s^2}+\T_{st^2s^2}+\T_{s^2t^2s}+\xi_1\T_{sts}+\xi_2\T_{ts^2t}+\xi_1\xi_2\T_{t^2st^2}+\xi_1(\T_{s^2t}+\T_{ts^2})+ \xi_1^2(\T_{sts^2}+\T_{s^2ts})$,}\\
(11)&\multicolumn{6}{p{.85\textwidth}}{ $\T_{ts^2}+\T_{s^2t}+\T_{s^2t^2s^2}+\xi_1\xi_2(\T_{s^2t}+\T_{ts^2})+(\xi_1+\xi_2^2)\T_{ts^2t} +\xi_2(1+\xi_1\xi_2)\T_{t^2st^2}+\xi_1(1+\xi_1\xi_2)(\T_{sts^2}+\T_{s^2ts})$,}\\
(12)&\multicolumn{6}{p{.85\textwidth}}{ $\T_{t^2s^2}+\T_{s^2t^2}+\T_{ts^2t}+\xi_1\xi_2\T_{ts^2t}+\xi_1^2\xi_2\T_{t^2st^2}+\xi_1^2\T_{sts}+ (\xi_1^2-\xi_2)(\T_{s^2t}+\T_{ts^2})+\xi_1^3(\T_{sts^2}+\T_{s^2ts})+\xi_1(\T_{st^2s^2}+\T_{s^2t^2s})$.}
\end{tabular}
\end{prop}

\begin{proof}
It suffices to check the conditions of Proposition~\ref{basis}.  That each of these centralizes $\T_t$ follows from their construction.  They each also specialize to a $t$-class sum in the group algebra.  It thus remains to check that each contains a distinguished element.

The first six terms are all singletons, and a quick check reveals they appear nowhere else (we removed them in preceding calculations).  For the remaining terms, the following elements respectively are distinguished: $\T_s$, $\T_{ts}$ or $\T_{st}$, $\T_{t^2s}$ or $\T_{st^2}$, $\T_{s^2}$, $\T_{s^2t^2s^2}$, and $\T_{t^2s^2}$ or $\T_{s^2t^2}$.
\end{proof}

We may now use the existence of this basis (with distinguished terms) for $Z_\H(\T_t)$ to determine necessary and sufficient relations between coefficients of terms in an element of the centralizer, with respect to the $sts$-basis: relations obtained via Theorem \ref{characterized} and Corollary~\ref{cor:coeffrelsg4} (modified by symmetry to be for $Z_\H(\T_t)$) are with respect to $B_{tst}$.

\begin{cor}\label{cor:t-elts:rels} Let $\H$ be the Hecke algebra of $G_4$ with basis $B_{sts}$.
The following group elements are distinguished (in the sense of \ref{basis}) in elements of the centralizer $Z_\H(\T_t)$:
\[\mathcal M=\{1,s,s^2,t,t^2,st,st^2,s^2t^2,st^2s,s^2ts^2,s^2t^2s^2,ststst\}.\]
Coefficients of remaining terms satisfy the following dependency relations on the coefficients of the distinguished terms:
  \[
  \begin{array}[t]{rr@{\,=\,}l}
  \rel{T1}&r_{s^2ts}=r_{sts^2}      &r_s+\xi_1 r_{st^2}+\xi_1^2r_{s^2}+\xi_1^3r_{s^2t^2}+\xi_1(1+\xi_1\xi_2)r_{s^2t^2s^2}\cr
  \rel{T2}&r_{s^2t^2s}=r_{st^2s^2}  &r_{s^2}+\xi_1r_{s^2t^2}\cr
  \rel{T3}&r_{sts}      &r_{st^2}+\xi_1r_{st}+\xi_1r_{s^2}+\xi_1^2r_{s^2t^2}-\xi_2r_s\cr
  \rel{T4}&r_{ts^2t}    &\xi_2r_{s^2}+(1+\xi_1\xi_2)r_{s^2t^2}+(\xi_1+\xi_2^2)r_{s^2t^2s^2}\cr
  \rel{T5}&r_{t^2st^2}  &r_{st}+\xi_2r_{st^2}+\xi_1\xi_2r_{s^2}+\xi_1^2\xi_2r_{s^2t^2}+\xi_2(1+\xi_1\xi_2)r_{s^2t^2s^2}\cr
  \rel{T6}&r_{ts^2}=r_{s^2t}        &\xi_1r_{s^2}+(1+\xi_1\xi_2)r_{s^2t^2s^2}+(\xi_1^2-\xi_2)r_{s^2t^2}\cr
  \rel{T7}&r_{st}                &r_{ts}\cr
  \rel{T8}&r_{st^2}                &r_{t^2s}\cr
  \rel{T9}&r_{s^2t}                &r_{ts^2}\cr
  \rel{T10}&r_{s^2t^2}                &r_{t^2s^2}.
  \end{array}
  \]
\end{cor}

  Proposition \ref{t-elts} and Corollary \ref{cor:t-elts:rels} exhibit an important departure from the analogy with the Hecke algebras of finite Coxeter groups.  The distinguished terms in $B_{sts}$ for $Z_\H(\T_s)$ (Cor. \ref{cor:coeffrelsg4}) and $B_{tst}$ for $Z_\H(\T_t)$ all were predictable shortest elements of $s$-conjugacy (resp. $t$-conjugacy) classes.  The relations among coefficients all preserved a dependency order consistent with the length function (even though the positivity found in the Coxeter group case is lost).  Once we convert the $Z_\H(\T_t)$-basis to being with respect to $B_{sts}$, we lose this dependency order.  Note that for instance $s^2t^2s^2$ is a distinguished term, and it is certainly not shortest in its $t$-conjugacy class $\{ts^2,s^2t,s^2t^2s^2\}$.

\subsection{The centre of the Hecke algebra}\label{sec:g4.centre}

\subsubsection{Via coefficient relationships}

We can combine the coefficient relationships in Corollaries \ref{cor:coeffrelsg4} and \ref{cor:t-elts:rels} to find necessary and sufficient relationships among coefficients of a central element.  Hecke algebra elements whose coefficients with respect to $B_{sts}$ satisfy all the relationships for $s$ and $t$ class elements must be in both centralizers and therefore the centre.

\begin{thm}
  Let $\H$ be the Hecke algebra of $G_4$ with basis $B_{sts}$.
  An element $h=\sum_{w\in G_4}r_{\bm w,B_{sts}}\T_{\bm w,B_{sts}}$ is in $Z(\H)$ if and only if the following relationships hold between coefficients $r_{\bm w,B_{sts}}$ (writing each $\bm w$ as an explicit reduced expression):
  \[
  \begin{array}[t]{rr@{\,=\,}l}
  \textsc{\tiny [C1]}&r_{s^2t}=r_{ts^2}&(1+\xi_1\xi_2)r_{t^2s}+\xi_1r_{s^2}+\xi_1(\xi_1+\xi_2^2)r_{t^2s^2}\cr
  \textsc{\tiny [C2]}&r_{sts}     &r_{t^2s}+\xi_1r_{ts}+\xi_1^2r_{t^2s^2}+\xi_1r_{s^2}-\xi_2r_{s}\cr
  \textsc{\tiny [C3]}&r_{t}       &r_s+\xi_1r_{t^2s}+\xi_1\xi_2r_{t^2s^2}\cr
  \textsc{\tiny [C4]}&r_{s^2ts}=r_{sts^2}   &r_s+\xi_1(2+\xi_1\xi_2)r_{t^2s}+\xi_1^2r_{s^2}+\xi_1(\xi_2+\xi_1\xi_2^2+\xi_1^2)r_{t^2s^2}\cr
  \textsc{\tiny [C5]}&r_{st^2s}   &(1-\xi_1\xi_2^2)r_{t^2s^2}+\xi_1(1-\xi_2)r_{t^2s}-\xi_2r_{s}\cr
  \textsc{\tiny [C6]}&r_{ts^2t}   &(1+2\xi_1\xi_2+\xi_2^2)r_{t^2s^2}+\xi_2r_{s^2}+(\xi_1+\xi_2^2)r_{t^2s}\cr
  \textsc{\tiny [C7]}&r_{s^2ts^2} &r_{ts}+\xi_2(1+\xi_1\xi_2)r_{t^2s}+\xi_1\xi_2r_{s^2}+\xi_1\xi_2(\xi_1+\xi_2^2)r_{t^2s^2}\cr
  \textsc{\tiny [C8]}&r_{t^2st^2} &r_{ts}+\xi_2(2+\xi_1\xi_2)r_{t^2s}+\xi_1\xi_2r_{s^2}+\xi_2(\xi_2+\xi_1\xi_2^2+\xi_1^2)r_{t^2s^2}\cr
  \textsc{\tiny [C9]}&r_{s^2t^2s}=r_{st^2s^2} &r_{s^2}+\xi_1r_{t^2s^2}\cr
  \textsc{\tiny [C10]}&r_{t^2}    &r_{s^2}\cr
  \textsc{\tiny [C11]}&r_{s^2t^2s^2}&r_{t^2s}+\xi_2r_{t^2s^2}\cr
  \textsc{\tiny [C12]}&r_{st}     &r_{ts}\cr
  \textsc{\tiny [C13]}&r_{st^2}   &s_{t^2s}\cr
  \textsc{\tiny [C14]}&r_{s^2t^2} &s_{t^2s^2}.
  \end{array}
  \]
  The set $\mathcal M=\{1,s,s^2,ts,t^2s,t^2s^2,tststs\}$ forms a distinguished set in the sense of Proposition \ref{basis}.  Consequently, $Z(\H)$ has an $R$-basis.
\end{thm}

\begin{proof}
These all follow from the relations for coefficients in elements of $Z_\H(\T_s)$ and $Z_\H(\T_t)$, given in Corollaries \ref{cor:coeffrelsg4} and \ref{cor:t-elts:rels}.  Thus, they are necessary for centrality.

\rel{C1} follows from substituting \rel{S6} into \rel{T6}.  \rel{C2} is \rel{T3} with \rel{T9}.  \rel{C3} follows from \rel{S3}, \rel{C2} and \rel{C1}.  \rel{C4} follows from \rel{T1} and \rel{S6}.  \rel{C5} follows from \rel{S3} and \rel{C3}.  \rel{C6} follows from \rel{T4} and \rel{S6}.  \rel{C7} follows from \rel{S5} and \rel{C1}.  \rel{C8} follows from \rel{T5} and \rel{S6}.  \rel{C9} is \rel{T2}.  \rel{C10} follows from \rel{S2} and \rel{T2}.  \rel{C11} is \rel{S6}.

This implies the existence of an $R$-basis for $Z(\H)$ by Proposition \ref{basis}.  The relations are sufficient for centrality because of the existence of an $R$-basis whose elements are defined by \rel{C1} to \rel{C14}.  That is, if $h\in\H$ has coefficients satisfying relations \rel{C1} to \rel{C14} then it can be written as a linear combination of the basis elements, and hence is central.
\end{proof}

Because the relationships of the theorem are sufficient conditions for centrality, we may obtain a central element for each $\bm w\in\mathcal M$ by setting $r_{\bm w}=1$ and $r_{\bm{w'}}=0$ if $\bm{w'}\in\mathcal M\setminus\{\bm w\}$, as in the proof of Proposition \ref{basis}.  We thus obtain an integral basis for $Z(\H)$:

\begin{cor}
The set $\B=\{b_1,\dots,b_7\}$ forms an $R$-basis for the centre of the Hecke algebra of $G_4$, where
\begin{align*}
    b_1 &=\T_1,\\ 
    b_2 &=\T_{ststst},\\
    b_3 &=\T_s+\T_t+T_{s^2ts}+\T_{sts^2}-\xi_2T_{sts} .\\
    b_4 &=\T_{ts}+\T_{st}+\T_{s^2ts^2}+\T_{t^2st^2}+\xi_1\T_{sts}\\ 
    b_5 &=\T_{s^2}+\T_{t^2}+\T_{st^2s^2}+\T_{s^2t^2s}+\xi_1^2(\T_{sts^2}+\T_{s^2ts})+\xi_1(\T_{s^2t}+\T_{ts^2})+\\
        &\quad\xi_1\T_{sts}+ \xi_2(\T_{ts^2t}-\T_{st^2s})+\xi_1\xi_2(\T_{t^2st^2}+\T_{s^2ts^2}). \\ 
    b_6 &=\T_{ts^2}+\T_{s^2t}+\T_{sts}+\T_{st^2}+\T_{t^2s}+\T_{s^2t^2s^2}+\xi_2\T_{t^2st^2}+ \xi_2^2\T_{ts^2t}+\xi_1\T_t+\\
        &\quad(2\xi_1+\xi_1^2\xi_2)(\T_{s^2ts}+\T_{sts^2})+\xi_1\xi_2(\T_{ts^2}+\T_{s^2t})+\xi_1(\T_{st^2s}+\T_{ts^2t})+\\
        &\quad\xi_2(1+\xi_1\xi_2)(\T_{s^2ts^2}+\T_{t^2st^2}).\\ 
    b_7 &=\T_{ts^2t}+\T_{s^2t^2}+\T_{t^2s^2}+\T_{st^2s}+\xi_2\T_{s^2t^2s^2}+\xi_1(\xi_1+\xi_2^2)(\T_{ts^2}+\T_{s^2t})+\\
        &\quad\xi_1(\xi_1^2+\xi_2+\xi_1\xi_2^2)(\T_{sts^2}+\T_{s^2ts})+\xi_1(\T_{st^2s^2}+\T_{s^2t^2s})+\xi_1\xi_2\T_t+ \\
        &\quad\xi_1^2\T_{sts}+\xi_2(2\xi_1+\xi_2^2)\T_{ts^2t}+\xi_2(\xi_2+\xi_1^2+ \xi_1\xi_2^2)\T_{t^2st^2}+\\
        &\quad\xi_1\xi_2(\xi_1+\xi_2^2)\T_{s^2ts^2} .
\end{align*}
\end{cor}
\begin{proof}
The sets $\B$ and $\mathcal M$ and the relations among coefficients all satisfy the requirements of Proposition \ref{basis}, by our construction.
\end{proof}

\subsubsection{Via bases for centralizers}

In \cite{Fmb}, an algorithm was given for constructing ``minimal'' central elements from ``minimal'' basis elements of the set of centralizers of generators in the context of Hecke algebras of finite Coxeter groups.  We have not introduced a partial order or positive cone in the context of the Hecke algebras of complex reflection groups but these centralizer basis elements are analogous to the minimal basis of the real reflection group case.  Indeed, central elements can be constructed from the bases for the centralizers in much the same way, however it does not appear at this stage that this construction is algorithmic.

In Figures \ref{fig:b3} to \ref{fig:b7} the class elements $b_3$ to $b_7$ are represented by diagrams whose nodes are the terms of the element (including coefficients) and whose edges connect terms in the same $s$-class (or $t$-class) element.  Dashed lines connect elements in a common $t$-class element, while solid lines connect elements in an $s$-class element.  In some cases elements cancel each other out, and these pairs are shown by being joined by a double-dashed line ``===".

\begin{figure}[ht]  
    \begin{minipage}[t]{.45\linewidth}
        \centering
        \scalebox{0.8}{\includegraphics{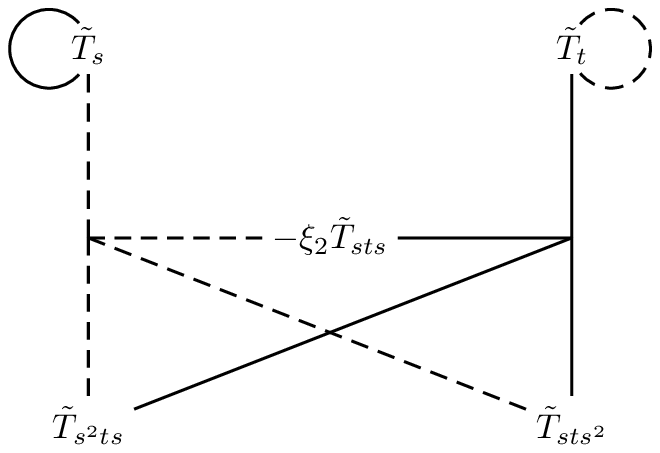}}%
        \caption{Diagram for the central basis element $b_3$ in type $G_4$ (Corollary \ref{cor:coeffrelsg4}).}\label{fig:b3}
    \end{minipage}%
    \begin{minipage}[t]{.05\linewidth}
    \text{}
    \end{minipage}
    \begin{minipage}[t]{.45\linewidth}
        \centering
        \scalebox{0.8}{\includegraphics{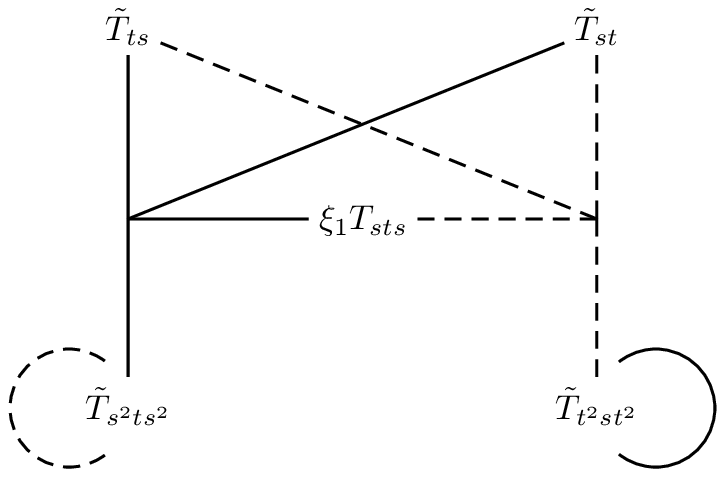}}%
        \caption{Diagram for the central basis element $b_4$ in type $G_4$ (Corollary \ref{cor:coeffrelsg4}).}\label{fig:b4}
    \end{minipage}
\end{figure}

\begin{figure}[h]  
\centerline{\scalebox{0.95}{\includegraphics{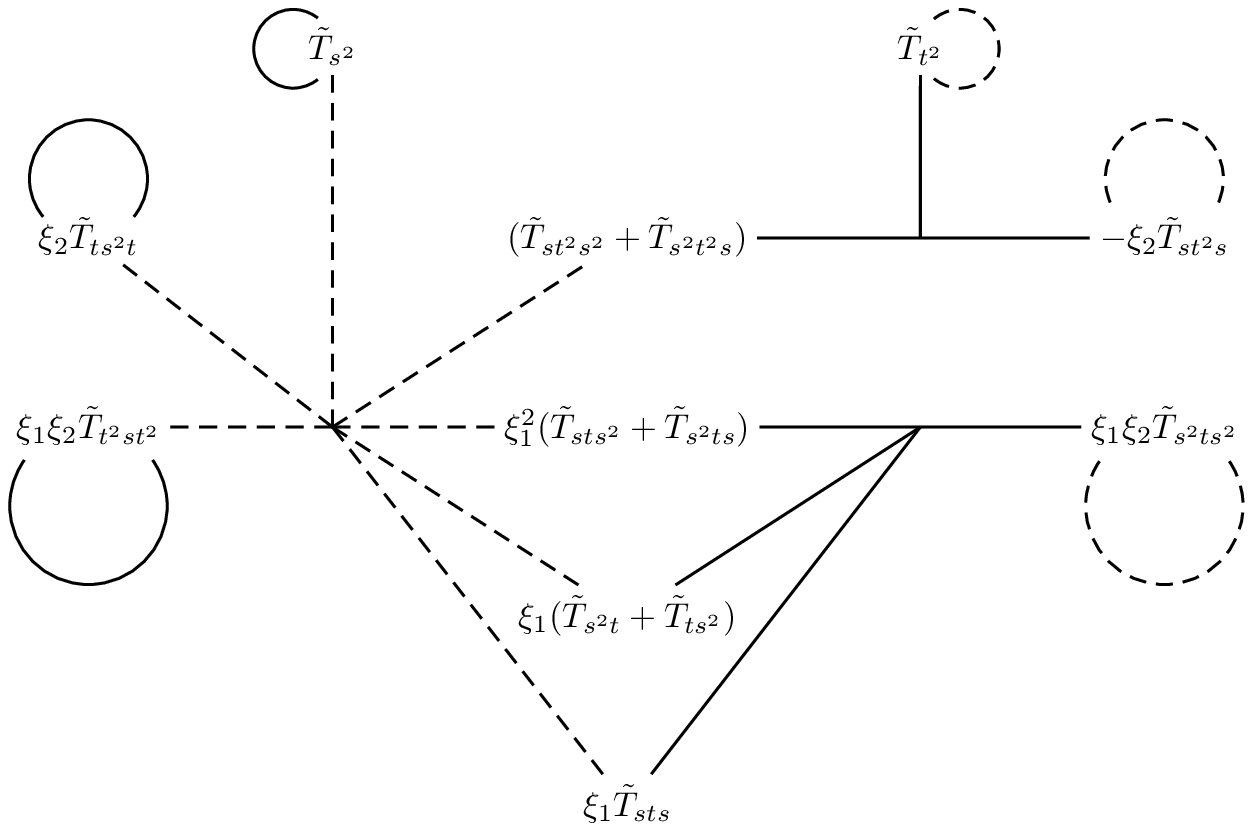}}}
\caption{Diagram for the central basis element $b_5$ in type $G_4$ (Corollary \ref{cor:coeffrelsg4}).}\label{fig:b5}
\end{figure}

\begin{figure}[h]  
\centering
\scalebox{0.85}{\includegraphics{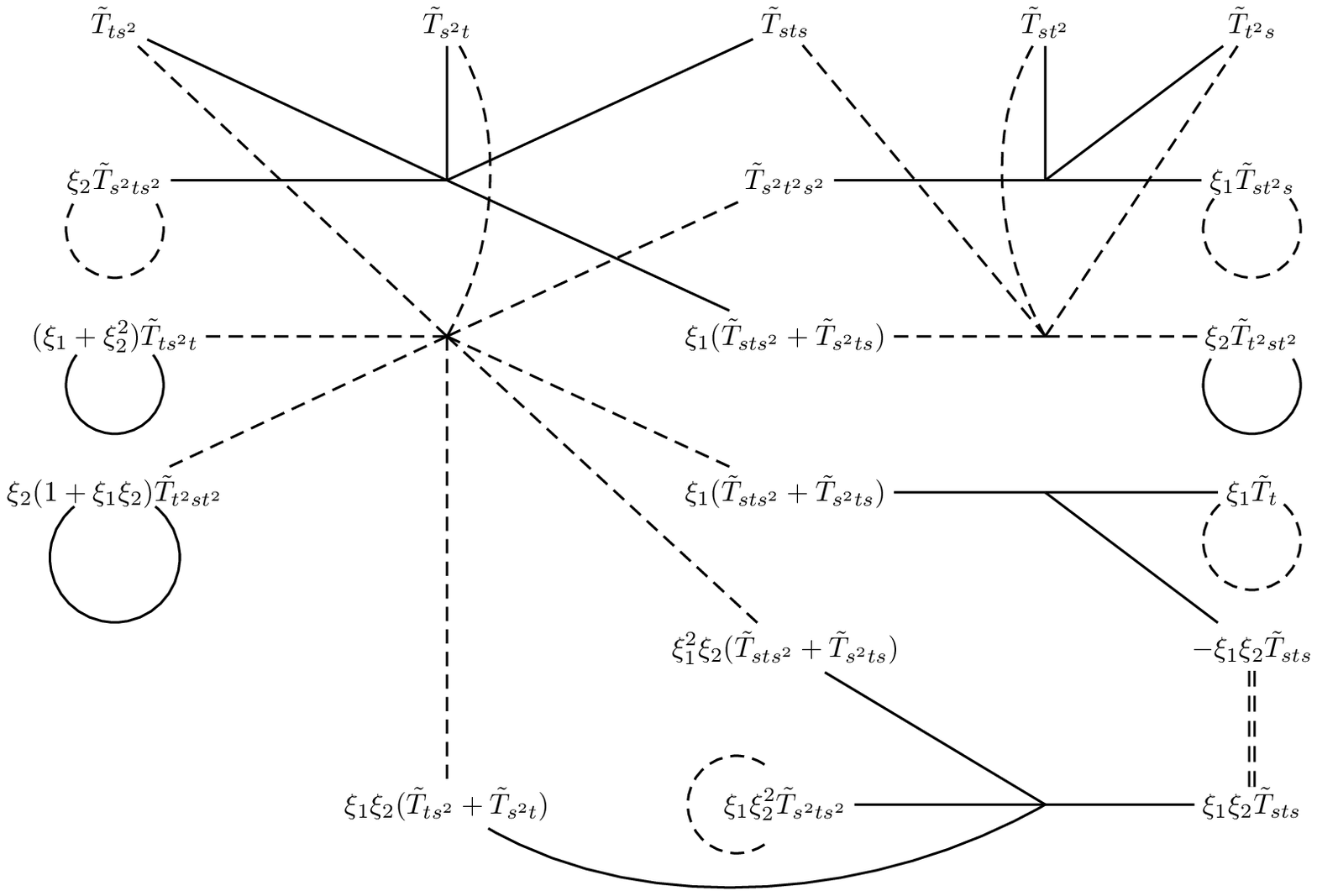}}
\caption{Diagram for the central basis element $b_6$ in type $G_4$ (Corollary \ref{cor:coeffrelsg4}).}\label{fig:b6}
\end{figure}

\begin{figure}[h]  
\centerline{\scalebox{0.7}{\includegraphics{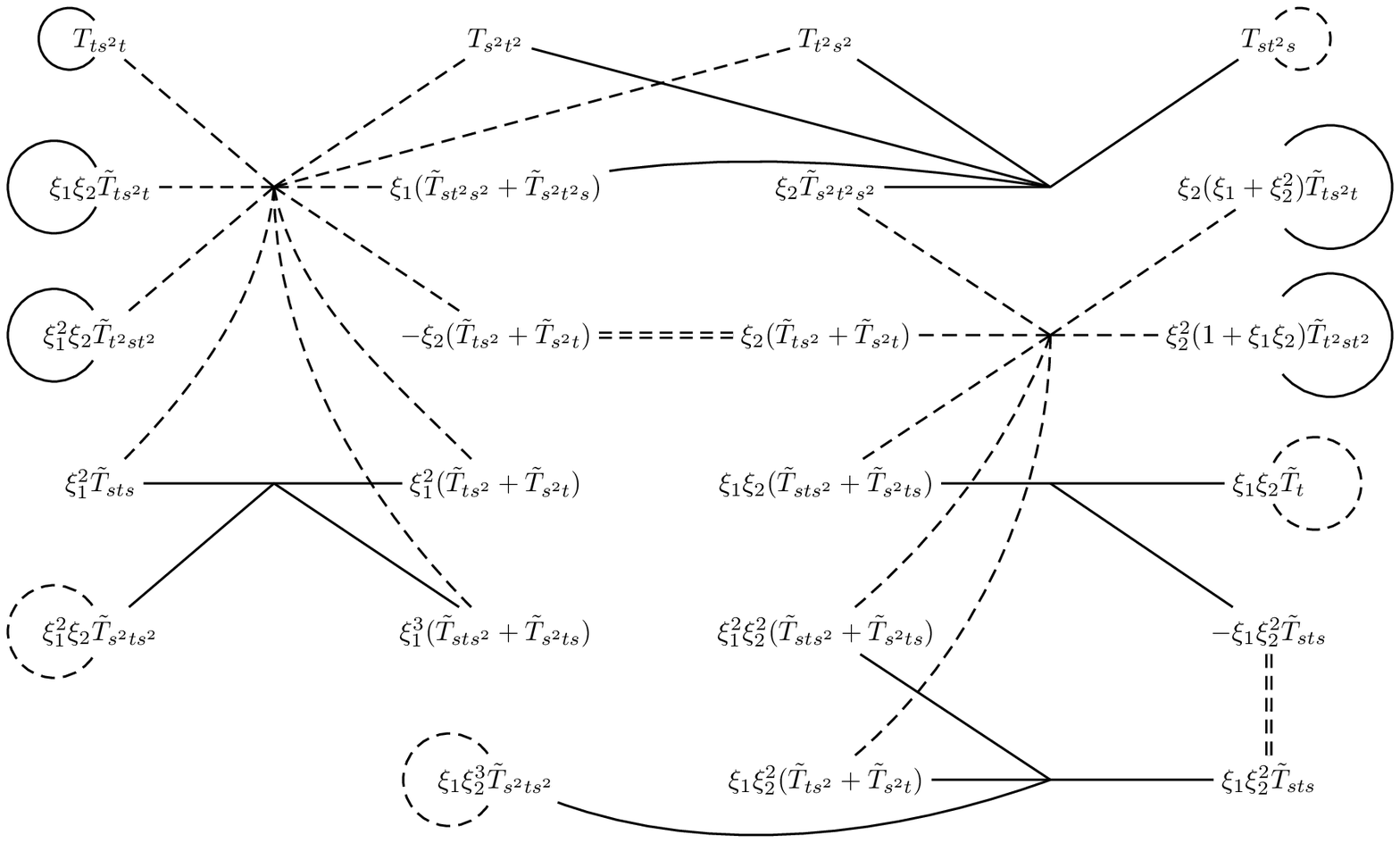}}}
\caption{Diagram for the central basis element $b_7$ in type $G_4$ (Corollary \ref{cor:coeffrelsg4}).}\label{fig:b7}
\end{figure}

\section{Types $G(r,1,n)$}\label{sec:Gr1n}

In this section we study the Hecke algebras of the family of groups $G(r,1,n)$.  Here not all double cosets are stable and it is necessary to consider the $\H$-double coset graph, as introduced in Section \ref{sec:bg:H.dcoset.graph}.  We describe completely the $\H$-double coset graph when $n=2$ in Section \ref{sec:gr12.d.coset.graphs}, and give integral bases for centralizers and the centre of the Hecke algebra of type $G(4,1,2)$ in Section \ref{sec:g412.bases}.

\subsection{Reduced expressions and relations}\label{sec:gr1n:reduced.exprs}

The generators and relations for the group $G(r,1,n)$ are given in Section \ref{sec:bg:cxrefgps}, and a Coxeter-type diagram is given in Figure \ref{fig:gr1n}.   The Hecke algebra here is defined over the ring $R=\Z[\xi_{t,i},\xi_{s}\mid 1\le i\le r-1]$; all $\{s_i\mid 1\le i\le n-1\}$ are conjugate so that the parameters $\xi_{s_i}=:\xi_s$ are equal.

Two elements of $W$ are said to be \emph{weakly braid equivalent} (\cite{BremMa97}) if they can be transformed into each other using just the braid relations and the relation $s_1t^as_1t^b=t^bs_1t^as_1$.

\begin{lem}[{\cite[(1.5)]{BremMa97}}]\label{lem:BM97:red_word}
  A reduced expression for any element of $G(r,1,n)$ is weakly braid equivalent to a (reduced) word of form
  \begin{equation}t_{0,a_0}\dots t_{n-1,a_{n-1}}\sigma\label{eq:lem:BM97:red_word}
  \end{equation}
  where $0\le a_i\le r-1$, $t_{k,a_i}:=s_k\dots s_1t^{a_i}$ for $0\le i\le n-1$ and $\sigma\in S_n$ is reduced.
\end{lem}

We will set the basis for $\H$ to be the set $B=\{\T_{\bm w,B}\}$ with $w$ in the form of \eqref{eq:lem:BM97:red_word}.
As with $G_4$, we will simplify notation by omitting reference to the basis $B$ in the subscript of a Hecke algebra term when the reduced expression is given explicitly.
To find bases for the centralizers with our approach we first need to establish which double cosets are stable, and of those, which are centralizing, additive, or neither.
The following will be useful:
\begin{lem}[{\cite[(2.2)]{BremMa97}}]\label{lem:BM97:red_word_diff}
Let $\H$ be the Hecke algebra of $G(r,1,n)$.  Then
\begin{equation}\label{eq:lem:BM97:red_word_diff}    \T_{s_1}\T_{t^a}\T_{s_1}\T_{t^b}=\T_{t^b}\T_{s_1}\T_{t^a}\T_{s_1}+\xi_s\sum_{i=1}^b(\T_{t}^{a+b-i}\T_{s_1}\T_{t^i}-\T_{t^i}\T_{s_1}\T_{t}^{a+b-i})
\end{equation}
for $1\le a\le r-1$ and $0\le b\le r-1$.
\end{lem}
Note that when $a+b\le r$, all terms on the right hand side of Eq.~\eqref{eq:lem:BM97:red_word_diff} are reduced.

\begin{lem}
  Let $\H$ be the Hecke algebra of $G(r,1,n)$ with reduced basis $B$.
  The double cosets $\l s_1\r t^as_1t^b\l s_1\r$ with $a>1$ are stable with respect to $B$ if and only if $b=0$ or $1$.
\end{lem}
\begin{proof}  
    In the case $b=0$, the double coset is additive with distinguished representative $t^a$ and the sum on the right hand side of Eq.~\eqref{eq:lem:BM97:red_word_diff} reduces to elements of $\l s_1\r t^a\l s_1\r$ (in the same double coset).  Its additivity together with Lemma \ref{lem:additive_implies_stable} gives its stability.

In the case $b=1$, the double coset is not additive.  However the sum on the right hand side of Eq.~\eqref{eq:lem:BM97:red_word_diff} has only one term $\T_{t}^{a}\T_{s_1}\T_{t}-\T_{t}\T_{s_1}\T_{t}^{a}$, which is in 
$\H_{\l s_1\r t^as_1t\l s_1\r}$.   There are several sets of reduced expressions for the four elements of $\l s_1\r t^is_1t\l s_1\r$, some of which are not equal in $\H$.  The reduced expressions are $t^is_1t$, $ts_1t^i$, $t^is_1ts_1=s_1ts_1t^i$, and $ts_1t^is_1=s_1t^is_1t$.  Whichever four of these is chosen as part of the reduced basis, the previous lemma together with the comments above imply stability.

If $b>1$ the sum on the right hand side of \eqref{eq:lem:BM97:red_word_diff} contains more than one term and for $i=1$ contains the expression $\T_{t}^{a+b-1}\T_{s_1}\T_{t}-\T_{t}\T_{s_1}\T_{t}^{a+b-1}$ whose terms are not in the double coset, and which is non-zero since $a+b>2$.  Instability follows immediately.
\end{proof}

Lemma \ref{lem:BM97:red_word_diff} is also useful for describing the $\H$-double coset graph when the double coset $\l s_1\r t^as_1t^b\l s_1\r$ is not stable.  We will return to this below.

\subsection{$\H$-double coset graphs in types $G(r,1,2)$}\label{sec:gr12.d.coset.graphs}

We focus now on the case $n=2$, setting $s:=s_1$ for the remainder of Section~\ref{sec:Gr1n}.  We will describe the double coset graphs for both $\l t\r$ and $\l s\r$. Fix the basis $B=\{\T_{t^ast^b},\T_{t^ast^bs}\mid 0\le a,b\le r-1\}$ of $\H$.

We begin with a description of minimal length double coset representatives, which follows immediately from the definition of the basis $B$ and Lemma \ref{lem:BM97:red_word}.

\begin{lem}The following form sets of minimal length representatives of double cosets in $G(r,1,2)$:
\begin{center}
    \begin{tabular}{rl}
    $\l t\r$-$\l t\r$ double cosets:&$\{1,s,st^as\mid 1\le a\le r-1\}$.\\
    $\l s\r$-$\l s\r$ double cosets:&$\{1,t^a,t^ast^b\mid 1\le b\le a\le r-1\}$.
    \end{tabular}
\end{center}
\end{lem}

\begin{prop}\label{prop:t-dcoset.graph.Gr12}
  The $\H$-double coset graph for $\l t\r$-$\l t\r$ double cosets in the Hecke algebra $\H$ of $G(r,1,2)$ is as in Figure \ref{fig:dcosetgraph:t:r12}.
\end{prop}

\begin{figure}[ht]
\centerline{\includegraphics{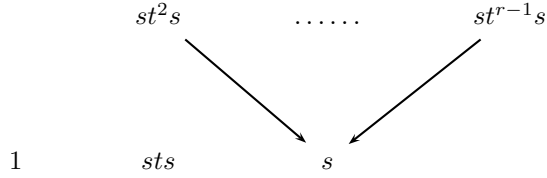}}
\caption{The $\H$-double coset graph for $\l t\r$-$\l t\r$ double cosets in $G(r,1,2)$.  Loops on vertices are not shown.}\label{fig:dcosetgraph:t:r12}
\end{figure}

\begin{proof}
    As the double cosets $\l t\r d\l t\r$ for $d=1,s,sts$ are either additive or centralizing, they are stable with respect to the basis $B$.  Thus they are terminal vertices in the $\H$-double coset graph.  It remains to consider $d=st^as$ for $a>1$.

  Because of our choice of basis, multiplication on the left does not generate terms in other double cosets: $\T_t\T_{st^as}=\T_{tst^as}$.  However multiplication on the right does.  (Note that the choice of reduced basis here does not affect the outcome).  It suffices to consider products of form $\T_{st^as}\T_t$ for $a>1$.  From Lemma \ref{lem:BM97:red_word_diff} we have
\[\T_{st^as}\T_t=\T_{tst^as}+\xi_s(\T_{t^ast}-\T_{tst^a}).\]
Since $a>1$ it follows that we have $\H_{\l t\r st^as\l t\r}\to\H_{\l t\r s\l t\r}$ and to no other double cosets.
\end{proof}

\begin{prop}\label{prop:s-dcoset.graph.Gr12}
Let $\H$ be the Hecke algebra of $G(r,1,2)$.
  The $\H$-double coset graph for $\l s\r$-$\l s\r$ double cosets has non-trivial arrows as follows, for $a\ge b\ge 1$:

\noindent{If $a+b\le r$, }$t^ast^b\longrightarrow t^{a+b-i}st^i$,{ for $1\le i\le b$,}

\noindent{If $a+b>r$, }$\displaystyle t^ast^b\longrightarrow
        \begin{cases} t^i,\ t^xst^i,& 1\le i\le a+b-r,\\
                                    & i<x\le r-1\\
                        t^{a+b-i}st^i,&a+b-r<i\le b.     \end{cases}$
\end{prop}

\begin{proof}
Firstly, note that the double cosets $\l s\r 1\l s\r$ and $\l s\r tst\l s\r$ are centralizing, and the double coset $\l s\r t^a\l s\r$ is additive.  Thus these are all stable, and the cases listed in the statement: $\l s\r t^ast^b\l s\r$ for $1\le b\le a\le r-1$ with $a>1$ are the only others remaining to consider.

In this case multiplication from the right leaves the double coset invariant: $\T_{t^ast^b}\T_s=\T_{t^ast^bs}\in\H_{\l s\r t^ast^b\l s\r}$, so it remains to consider the left multiplication, and again we will use Lemma \ref{lem:BM97:red_word_diff}.

  In the case $a+b\le r$, we have
  \begin{align*}\T_s\T_{t^ast^b}&=\T_{t^bst^as}+\xi_s\sum_{i=1}^b(\T_{t}^{a+b-i}\T_{st^i}-\T_{t^is}\T_t^{a+b-i})\\
                                &=\T_{t^bst^as}+\xi_s\sum_{i=1}^b(\T_{t^{a+b-i}st^i}-\T_{t^ist^{a+b-i}})
  \end{align*}
  since $a+b-i<r$.  There will be cancelling in the terms $\T_{t^{a+b-i}st^i}-\T_{t^ist^{a+b-i}}$ if and only if $a+b-i=i$.  But since $i\le b\le a$ this can only happen if $a=b$ (since otherwise $a+b-i\ge a>b\ge i$).  In this case the terms $\T_{t^ast^a}-\T_{t^ast^a}$ cancel but their double coset is still represented by $\T_{t^ast^as}$ and the statement holds.

In the case $a+b>r$ it is slightly more complicated as the exponents $a+b-i$ in Eq.~\eqref{eq:lem:BM97:red_word_diff} may be greater than $r-1$.  If $a+b-i\ge r$ then $\T_{t}^{a+b-i}$ will be a linear combination of elements of $\mathcal T=\{1,\T_t,\dots,\T_{t^{r-1}}\}$.  Note that the coefficient of each element of $\mathcal T$ will be non-zero in the expansion of $\T_{t}^{a+b-i}$, as each coefficient will be a \emph{positive} linear combination of the $\xi_{t,i}$.
Thus if $a+b>r$ we have for non-zero $a_i\in\N[\xi_{t,0},\dots,\xi_{t,r-1}]$,
\begin{align*}
\T_s\T_{t^ast^b}&=\T_{t^bst^as}+\xi_s\sum_{i=1}^b\left(\T_{t}^{a+b-i}\T_{st^i}-\T_{t^is}\T_t^{a+b-i}\right)\cr
&=\T_{t^bst^as}\cr
&\qquad+\xi_s\sum_{i=1}^{a+b-r}\left[\left(a_0+a_1\T_t+\dots+a_{r-1}\T_{t^{r-1}}\right)\T_{st^i}\right.\cr &\hspace{3cm}\left.-\T_{t^is}\left(a_0+a_1\T_t+\dots+a_{r-1}\T_{t^{r-1}}\right)\right]\cr
&\qquad+\xi_s\sum_{i=a+b-(r-1)}^b\left(\T_{t}^{a+b-i}\T_{st^i}-\T_{t^is}\T_t^{a+b-i}\right)\cr
&=\T_{t^bst^as}\cr
&\qquad+\xi_s\sum_{i=1}^{a+b-r}
\left[\left(a_0\T_{st^i}+a_1\T_{tst^i}+\dots+a_{r-1}\T_{t^{r-1}st^i}\right)\right. \cr
&\hspace{3cm}\left. -\left(a_0\T_{t^is}+a_1\T_{t^ist}+\dots+a_{r-1}\T_{t^ist^{r-1}}\right)\right]\cr
&\qquad+\xi_s\sum_{i=a+b-(r-1)}^b\left(\T_{t^{a+b-i}st^i}-\T_{t^ist^{a+b-i}}\right)\\
&=\T_{t^bst^as}\\
&\quad+\xi_s\left[\sum_{i=1}^{a+b-r}\left(a_0\T_{st^i}+\dots+a_{r-1}\T_{t^{r-1}st^i}\right)+\sum_{i=a+b-r+1}^b\T_{t^{a+b-i}st^i}\right]\\
&\quad-\xi_s\left[\sum_{i=1}^{a+b-r}\left(a_0\T_{t^is}+\dots+a_{r-1}\T_{t^ist^{r-1}}\right)+\sum_{i=a+b-r+1}^b\T_{t^ist^{a+b-i}}\right]
\end{align*}

It remains to check whether any terms cancel in these sums of differences.  To ease description, let \[L_+:=\sum_{i=1}^{a+b-r}\left(a_0\T_{st^i}+\dots+a_{r-1}\T_{t^{r-1}st^i}\right),\quad R_+:=\sum_{i=a+b-r+1}^b\T_{t^{a+b-i}st^i}\]
and
\[L_-:=\sum_{i=1}^{a+b-r}\left(a_0\T_{t^is}+\dots+a_{r-1}\T_{t^ist^{r-1}}\right),\quad R_-:=\sum_{i=a+b-r+1}^b\T_{t^ist^{a+b-i}}.\]
Then we have shown above that
\[\T_s\T_{t^ast^b}=\T_{t^bst^as}+(L_++R_+)-(L_-+R_-),\]
noting that all of $\{L_-,L_+,R_-,R_+\}$ are positive (in $\H^+$).

It is not possible for a term in the $L_+$ to cancel with a term in $R_-$, because to do so the exponent sum of the $t$'s in a  term in $L_+$ would have to be $a+b$, but it is bounded above by $r-1+(a+b-r)=a+b-1$.  Similarly, terms in $R_+$ and $L_-$ cannot cancel, and we only need consider possible cancellations between $L_+$ and $L_-$ on the one hand, and $R_+$ and $R_-$ on the other.

We start with the latter.  Note that $i\le a+b-i$ since $i\le b$ and $b\le a$.  Thus, these terms will only cancel if for some $i$, $i=a+b-i$.  In this case $2i=a+b$.  Since $i\le b$ and $b\le a$, we have $2i\le 2b$ and so the only solution is $i=a=b$.  In this situation despite the cancellation we have the term $\T_{t^bst^as}=\T_{t^ast^as}$ in the expansion so we still have $t^ast^a\to t^ast^a$ (that is, we haven't lost representation from any double coset, despite the cancelling).  Thus we have $t^ast^b\to t^{a+b-i}st^i$ with $a+b-r<i\le b$.

Now the former.  By Lemma \ref{lem:power.of.Ts}, we have that the $a_i$ are all distinct, and therefore the only cancellation possible is $a_i\T_{t^ist^i}-a_i\T_{t^ist^i}$.  This leaves the remaining terms intact, and because we have $\l s\r t^ist^j\l s\r=\l s\r t^jst^i\l s\r$, we therefore also have $t^ast^b\to t^i$ for $1\le i\le a+b-r$ and $t^ast^b\to t^xst^i$ for $i<x\le r-1$.  This completes the proof.
\end{proof}

This completes the description of the $\H$-double coset graphs in type $G(r,1,2)$ for $W_J$-$W_J$ double cosets where $J\subseteq S$ and $|J|=1$.

Of the double cosets in $G(r,1,2)$, it follows that only $\l t\r s\l t\r$, $\l t\r 1\l t\r$, $\l t\r sts\l t\r$, $\l s\r t^i\l s\r$ and $\l s\r t^ist\l s\r$ for $0\le i\le r-1$ are stable.  For all others, to evaluate relationships among coefficients on elements of the centralizer of a generator, it is necessary to consider linear combinations of a set of double cosets.  As can be seen from the above proposition, this may be quite a large task.

\begin{rem}
Observe that the $\H$-double coset graphs in the Hecke algebra of $G(r,1,2)$ are partially ordered sets.  In greater generality, while the reflexivity of the relation is clear, antisymmetry and transitivity are not because they depend on the relations in the Hecke algebra in question (and may also depend on the choice of basis).
\end{rem}

\subsubsection{Example: $\H$-double coset graphs in type $G(4,1,2)$.}\label{sec:g412.d.coset.graph}

In $G(4,1,2)$ the $\l s\r$-$\l s\r$ double coset representatives are
\[
\{1,t,t^2,t^3,tst,t^2st,t^3st,t^2st^2,t^3st^2,t^3st^3\}.
\]
Thus for $t^ast^b$, we have $a+b> 4=r$ for the representatives $t^3st^2$, $t^3st^3$.  It follows from Proposition \ref{prop:s-dcoset.graph.Gr12} that the non-trivial arrows in the graph are:
\begin{align*}
  t^3st^3&\to t,t^2st,t^3st,t^2,t^3st^2\\
  t^3st^2&\to t,t^2st,t^3st\\
  t^2st^2&\to t^3st.
\end{align*}

The graph therefore (drawn as a poset, eliminating redundant edges) is given in Figure \ref{fig:g412.d.coset.graph}.

\begin{figure}[ht]
\centering\includegraphics{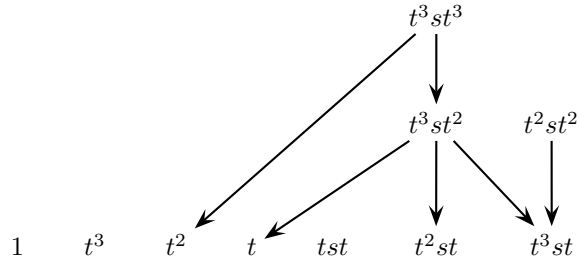}
\caption{$\H$-double coset graph for $\l s\r$-$\l s\r$ double cosets in the Hecke algebra of $G(4,1,2)$, drawn as a poset.}\label{fig:g412.d.coset.graph}
\end{figure}

\subsection{Centres and centralizers in type $G(4,1,2)$}\label{sec:g412.bases}

For low $r$ the results of Lemma \ref{prop:t-dcoset.graph.Gr12} and Proposition \ref{prop:s-dcoset.graph.Gr12} are less daunting and we have used them to find relationships among coefficients of elements of the centralizers.  We have then used these relationships to find relationships among coefficients in central elements, and consequently have found integral bases for the centre of the Hecke algebras in types $G(3,1,2)$ and $G(4,1,2)$.

We give as an example the bases of the centralizers and centre for the Hecke algebra of $G(4,1,2)$.
As above, write $s:=s_1$, so that \[G(4,1,2)=\l s,t\mid s^2=t^4=1,stst=tsts\r.\]
Note that as with $G_4$ there is a choice of basis for $\H$ involved, and we have chosen one consistent with the form in Lemma \ref{lem:BM97:red_word}, namely
\begin{multline*}
\{\T_{1},\T_{s},\T_{t},\T_{t^2},\T_{t^3},\T_{ts},\T_{st},\T_{st^2},\T_{t^2s},\T_{st^3},\T_{t^3s},\T_{tst},\T_{t^2st},\T_{tst^2},\\ \T_{t^2st^2},\T_{t^3st},\T_{tst^3},\T_{t^2st^3},\T_{t^3st^2},\T_{t^3st^3},\T_{sts},\T_{tsts},\T_{t^2sts},\T_{t^3sts},\\
\T_{st^2s},\T_{tst^2s},\T_{t^2st^2s},  \T_{t^3st^2s},\T_{st^3s},\T_{tst^3s},\T_{t^2st^3s},\T_{t^3st^3s}\}.
\end{multline*}
There is a set of transformations between bases similar to that for the Hecke algebra of $G_4$ relating for instance $\T_{st^2st^2}$, which is not in the basis, with terms which are in the basis.  We omit this for brevity.

\def\xc {{\xi_{t,3}}}
\def\xb {{\xi_{t,2}}}
\def\xa {{\xi_{t,1}}}
\def\xs {{\xi_s}}

The $t$-class and $s$-class elements are given below.
These were constructed in a manner very similar to the case of $G_4$ in the case of those double cosets which are stable.  However for non-stable double cosets the double coset graph was required, so that for instance relationships among coefficients for the $s$-class containing $t^2st^2$ needed to be found by considering a linear combination of terms from the basis elements in $\H_{\l s\r t^2st^2\l s\r}\cup\H_{\l s\r t^3st\l s\r}$.

Observe that the first term in each $t$-class (resp. $s$-class) element below is distinguished, in the sense (of Proposition~\ref{basis}) that it appears uniquely in that $t$-class (resp. $s$-class) element.

The $t$-class elements:
\[\begin{array}{lllll}
{\rel{T1}}: \T_1    &{\rel{T2}}: \T_t       & {\rel{T3}}: \T_{t^2}  &{\rel{T4}}: \T_{t^3}\\
{\rel{T5}: \T_{sts}}&{\rel{T6}}: \T_{tsts}  &{\rel{T7}}: \T_{t^2sts}&{\rel{T8}}: \T_{t^3sts}\\
\multicolumn{2}{l}{\rel{T9}: \T_{st^2s}+\xs\T_{tst}}&\multicolumn{2}{l}{\rel{T10}: \T_{tst^2s}+\xs\T_{t^2st}}&\multicolumn{1}{l}{\rel{T11}: \T_{t^2st^2s}+\xs\T_{t^3st}}\\
\multicolumn{5}{l}{\rel{T12}: \T_{t^3st^2s}+\xs\T_{st}+\xs\xa\T_{tst}+\xs\xb\T_{t^2st}+\xs\xc\T_{t^3st}}\\
\multicolumn{3}{l}{\rel{T13}: \T_{st^3s}+\xs\T_{t^2st}+\xs\T_{tst^2}}&\multicolumn{2}{l}{\rel{T14}: \T_{tst^3s}+\xs\T_{t^3st}+\xs\T_{t^2st^2}}\\
\multicolumn{5}{l}{\rel{T15}: \T_{t^2st^3s}+\xs\T_{st}+\xs\xa\T_{tst}+\xs\xb\T_{t^2st}+\xs\xc\T_{t^3st}+\xs\T_{t^3st^2}}\\
\multicolumn{5}{l}{\rel{T16}: \T_{t^3st^3s}+\xs\xc\T_{st}+\xs\T_{st^2}+\xs(1+\xa\xc)\T_{tst}+\xs\xa\T_{tst^2}+}\\
\multicolumn{5}{l}{\hspace{1cm} \xs(\xa+\xb\xc)\T_{t^2st}+\xs(\xb+\xi_{t,3}^2)\T_{t^3st}+\xs\xb\T_{t^2st^2}+\xs\xc\T_{t^3st^2}}\\
\multicolumn{5}{l}{\rel{T17}: \T_s+\T_{t^3st}+\T_{tst^3}+\T_{t^2st^2}-\xb\T_{tst}-\xc\T_{t^2st}-\xc\T_{tst^2}}\\
\multicolumn{5}{l}{\rel{T18}: \T_{ts}+\T_{st}+\T_{t^3st^2}+\T_{t^2st^3}+\xa\T_{tst}-\xc\T_{t^2st^2}}\\
\multicolumn{5}{l}{\rel{T19}: \T_{t^2s}+\T_{st^2}+\T_{tst}+\T_{t^3st^3}+\xa\T_{t^2st}+\xa\T_{tst^2}+\xb\T_{t^2st^2}}\\
\multicolumn{5}{l}{\rel{T20}: \T_{t^3s}+\T_{st^3}+\T_{t^2st}+\T_{tst^2}+\xa\T_{tst^3}+\xa\T_{t^3st}+\xa\T_{t^2st^2}+}\\
\multicolumn{5}{l}{\hspace{1cm}\xb\T_{t^2st^3}+\xb\T_{t^3st^2}+\xc\T_{t^3st^3}.}
\end{array}\]

The  $s$-class elements:
\[\begin{array}{lllll}
\rel{S1}: \T_1&\rel{S2}: \T_s\quad &\rel{S3}: \T_{tst}&\rel{S4}: \T_{tsts}\\
\multicolumn{3}{l}{\rel{S5}: \T_{t^2st^2}+\xs\T_{t^3sts}}&\multicolumn{2}{l}{\rel{S6}: \T_{t^2st^2s}+\xs\T_{t^3st}+\xi_s^2\T_{t^3sts}}\\
\multicolumn{5}{l}{\rel{S7}: \T_{t^3st^3}+\xs\xc\T_{t^3st^2s}+\xi_s^2\xi_{t,3}^2\T_{t^3st}+\xs(\xb+\xi_{t,3}^2(1+\xi_s^2))\T_{t^3sts}+}\\
\multicolumn{5}{l}{\hspace{1cm}\xi_s^2\xb\xc\T_{t^2st}+\xs\xb\xc(1+\xi_s^2)\T_{t^2sts}+\xs\T_{st^2s}+\xi_s^2\xc\T_{st}+}\\
\multicolumn{5}{l}{\hspace{1cm}\xs\xc(1+\xi_s^2)\T_{sts}}\\
\multicolumn{5}{l}{\rel{S8}: \T_{t^3st^3s}+\xs\xc\T_{t^3st^2}+\xi_s^2\xc\T_{t^3st^2s}+\xs(\xb+\xi_{t,3}^2(1+\xi_s^2))\T_{t^3st}+}\\
\multicolumn{5}{l}{\hspace{1cm}\xi_s^2(\xb+\xi_{t,3}^2(2+\xi_s^2))\T_{t^3sts}+\xs\xb\xc(1+\xi_s^2)\T_{t^2st}+}\\
\multicolumn{5}{l}{\hspace{1cm}\xi_s^2\xb\xc(2+\xi_s^2)\T_{t^2sts}+\xi_s^2\T_{st^2s}+\xs\T_{st^2}+\xi_s^2\xc(2+\xi_s^2)\T_{sts}+}\\
\multicolumn{5}{l}{\hspace{1cm}\xs\xc(1+\xi_s^2)\T_{st}}\\
\multicolumn{1}{l}{\rel{S9}: \T_t+\T_{sts}}&
\multicolumn{2}{l}{\rel{S10}: \T_{ts}+\T_{st}+\xs\T_{sts}}&
\multicolumn{2}{l}{\rel{S11}: \T_{t^2}+\T_{st^2s}}\\
\multicolumn{3}{l}{\rel{S12}: \T_{t^2s}+\T_{st^2}+\xs\T_{st^2s}}&
\multicolumn{2}{l}{\rel{S13}: \T_{t^3}+\T_{st^3s}}\\
\multicolumn{3}{l}{\rel{S14}: \T_{t^3s}+\T_{st^3}+\xs\T_{st^3s}}&
\multicolumn{2}{l}{\rel{S15}: \T_{tst^2}+\T_{t^2st}+\xs\T_{t^2sts}}\\
\multicolumn{3}{l}{\rel{S16}: \T_{tst^2s}+\T_{t^2sts}+\xs\T_{t^2st}+\xi_s^2\T_{t^2sts},}&
\multicolumn{2}{l}{\rel{S17}: \T_{tst^3}+\T_{t^3st}+\xs\T_{t^3sts},}\\
\multicolumn{5}{l}{\rel{S18}: \T_{tst^3s}+\T_{t^3sts}+\xs\T_{t^3st}+\xi_s^2\T_{t^3sts},}\\
\multicolumn{5}{l}{\rel{S19}: \T_{t^2st^3}+\T_{t^3st^2}+\xs\T_{t^3st^2s}+\xi_s^2\xc\T_{t^3st}+\xs\xc(2+\xi_s^2)\T_{t^3sts}+}\\
\multicolumn{5}{l}{\hspace{1cm}\xi_s^2\xb\T_{t^2st}+\xs\xb(2+\xi_s^2)\T_{t^2sts}+\xs(2+\xi_s^2)\T_{sts}+\xi_s^2\T_{st},}\\
\multicolumn{5}{l}{\rel{S20}: \T_{t^2st^3s}+\T_{t^3st^2s}+\xs\T_{t^3st^2}+\xi_s^2\T_{t^3st^2s}+\xs\xc(2+\xi_s^2)\T_{t^3st}+}\\
\multicolumn{5}{l}{\hspace{1cm}\xi_s^2\xc(3+\xi_s^2)\T_{t^3sts}+\xs\xb(2+\xi_s^2)\T_{t^2st}+}\\
\multicolumn{5}{l}{\hspace{1cm}\xi_s^2\xb(3+\xi_s^2)\T_{t^2sts}+\xi_s^2(3+\xi_s^2)\T_{sts}+\xs(2+\xi_s^2)\T_{st}.}
\end{array}\]

These bases of centralisers may be combined to form an integral basis for the centre.
We write {{[Ci]}} for the class element corresponding to the conjugacy class $C_i$.  The centre has an integral basis consisting of the following fourteen class elements:

\[
\begin{array}{lll}
\rel{C1}: \T_1&\qquad\rel{C2}: \T_{tsts}&\qquad\rel{C3}: \T_{t^2st^2s}+\xs\T_{t^3st}+\xi_s^2\T_{t^3sts}\\
\multicolumn{3}{l}{\rel{C4}: \T_{t^3st^3s}+\xs\xc\T_{t^3st^2}+\xi_s^2\xc\T_{t^3st^2s}+\xs\left(\xb+\xi_{t,3}^2(1+\xi_s^2)\right)\T_{t^3st}+}\\
\multicolumn{3}{l}{\hspace{.5cm}\xi_s^2\left(2\xb+\xi_{t,3}^2(2+\xi_s^2)\right)\T_{t^3sts}+\xs\xb\T_{t^2st^2}+\xi_s^2\T_{st^2s}+\xs\T_{st^2}+}\\
\multicolumn{3}{l}{\hspace{.5cm}\xs\left(\xa+\xb\xc(1+\xi_s^2)\right)\T_{t^2st}+\xi_s^2\left(\xa+\xb\xc(2+\xi_s^2)\right)\T_{t^2sts}+}\\
\multicolumn{3}{l}{\hspace{.5cm}\xs\xc(1+\xi_s^2)\T_{st}+\xi_s^2\xc(2+\xi_s^2)\T_{sts}+\xs(1+\xa\xc)(1+\xi_s^2)\T_{tst},}\\
\multicolumn{2}{l}{\rel{C5}: \T_t+\T_{sts}}&\qquad{\rel{C6}: \T_{t^2}+\T_{st^2s}+\xs\T_{tst}}\\
\multicolumn{3}{l}{\rel{C7}: \T_{t^3}+\T_{st^3s}+\xs\T_{t^2st}+\xs\T_{tst^2}+\xi_s^2\T_{t^2sts}}\\
\multicolumn{3}{l}{\rel{C8}: \T_{tst^2s}+\T_{t^2sts}+\xs\T_{t^2st}+\xi_s^2\T_{t^2sts}}\\
\multicolumn{3}{l}{\rel{C9}: \T_{tst^3s}+\T_{t^3sts}+\xs\T_{t^3st}+\xs\T_{t^2st^2}+2\xi_s^2\T_{t^3sts}}\\
\multicolumn{3}{l}{\rel{C10}: \T_{t^2st^3s}+\T_{t^3st^2s}+\xs\T_{t^3st^2}+\xi_s^2\T_{t^3st^2s}+\xs\xc(2+\xi_s^2)\T_{t^3st}+}\\
\multicolumn{3}{l}{\hspace{.5cm}\xi_s^2\xc(3+\xi_s^2)\T_{t^3sts}+\xs\xb(2+\xi_s^2)\T_{t^2st}+\xi_s^2\xb(3+\xi_s^2)\T_{t^2sts}+}\\
\multicolumn{3}{l}{\hspace{.5cm}\xi_s^2(3+\xi_s^2)\T_{sts}+\xs\xa(2+\xi_s^2)\T_{tst}+\xs(2+\xi_s^2)\T_{st}}\\
\multicolumn{3}{l}{\rel{C11}: \T_s+\T_{tst^3}+\T_{t^3st}+\T_{t^2st^2}+2\xs\T_{t^3sts}-\xb\T_{tst}-\xc\T_{t^2st}-}\\
\multicolumn{3}{l}{\hspace{.5cm}\xc\T_{tst^2}-\xs\xc\T_{t^2sts}}\\
\multicolumn{3}{l}{\rel{C12}: \T_{ts}+\T_{st}+\T_{t^3st^2}+\T_{t^2st^3}+\xs\T_{t^3st^2s}+\xi_s^2\xc\T_{t^3st}+\xi_s^2\T_{st}+}\\
\multicolumn{3}{l}{\hspace{.5cm}\xs\xc(1+\xi_s^2)\T_{t^3sts}+\xi_s^2\xb\T_{t^2st}+\xs\xb(2+\xi_s^2)\T_{t^2sts}-\xc\T_{t^2st^2}+}\\
\multicolumn{3}{l}{\hspace{.5cm}\xs(3+\xi_s^2)\T_{sts}+\xa(1+\xi_s^2)\T_{tst}}\\
\multicolumn{3}{l}{\rel{C13}: \T_{t^2s}+\T_{st^2}+\T_{tst}+\T_{t^3st^3}+\xs\xc\T_{t^3st^2s}+\xi_s^2\xi_{t,3}^2\T_{t^3st}+}\\
\multicolumn{3}{l}{\hspace{.5cm}\xs\xb\left(1+\xc(1+\xi_s^2)\right)\T_{t^3sts}+\xa\T_{tst^2}+(\xa+\xi_s^2\xb\xc)\T_{t^2st}+}\\
\multicolumn{3}{l}{\hspace{.5cm}\xs\left(\xa+\xb\xc(1+\xi_s^2)\right)\T_{t^2sts}+\xb\T_{t^2st^2}+2\xs\T_{st^2s}+\xi_s^2\xc\T_{st}+}\\
\multicolumn{3}{l}{\hspace{.5cm}\xi_s^2(2+\xa\xc)\T_{tst}+\xs\xc(1+\xi_s^2)\T_{sts}}\\
\multicolumn{3}{l}{\rel{C14}: \T_{t^3s}+\T_{st^3}+\T_{tst^2}+\T_{t^2st}+\xc\T_{t^3st^3}+\xa\T_{t^2st^2}+\xs\T_{st^3s}+}\\
\multicolumn{3}{l}{\hspace{.5cm}\xs\xc\T_{st^2s}+\xs(\xb+\xi_{t,3}^2)\T_{t^3st^2s}+\xb\T_{t^3st^2}+\xb\T_{t^2st^3}+}\\
\multicolumn{3}{l}{\hspace{.5cm}\xs\left(2\xa+2\xb\xc+\xc(\xb+\xi_{t,3}^2)(1+\xi_s^2)\right)\T_{t^3sts}+\xa\T_{tst^3}+}\\
\multicolumn{3}{l}{\hspace{.5cm}\left(\xa+\xi_s^2\xc(\xb+\xi_{t,3}^2)\right)\T_{t^3st}+\xi_s^2(\xb+\xi_{t,3}^2)\T_{st}+}\\
\multicolumn{3}{l}{\hspace{.5cm}\xi_s^2\left(1+\xb(\xb+\xi_{t,3}^2)\right)\T_{t^2st}+\xs\xi_{t,3}^2\left(1+\xi_s^2\right)\T_{sts}+}\\
\multicolumn{3}{l}{\hspace{.5cm}\xi_s^2\left(\xc+\xa(\xb+\xi_{t,3}^2)\right)\T_{tst}}
\end{array}
\]

Note that once again each element has a distinguished term, which is the first as written above.

Figures \ref{fig:g412.c9} and \ref{fig:g412.c11} show some of the structure of some of the class elements, and incidentally constitute proofs that {\rel{C9}} and {\rel{C11}} are central, as they show how each can be written as a linear combination of $s$ and $t$-class elements.

\begin{figure}[ht]
\centerline{\includegraphics{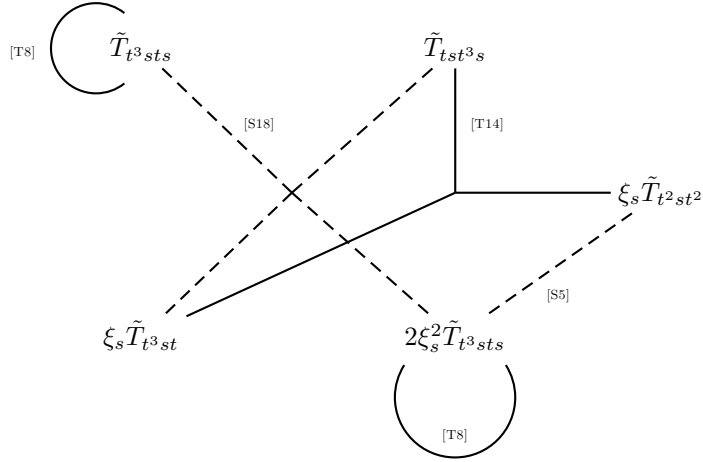}}
\caption{ Diagram for the class element {[C9]} in the Hecke algebra of $G(4,1,2)$. The solid edges indicate terms in the same $t$-class element (as labelled), and the dashed edges indicate terms in the same $s$-class element.}\label{fig:g412.c9}
\end{figure}

\begin{figure}[ht]
\centerline{\includegraphics{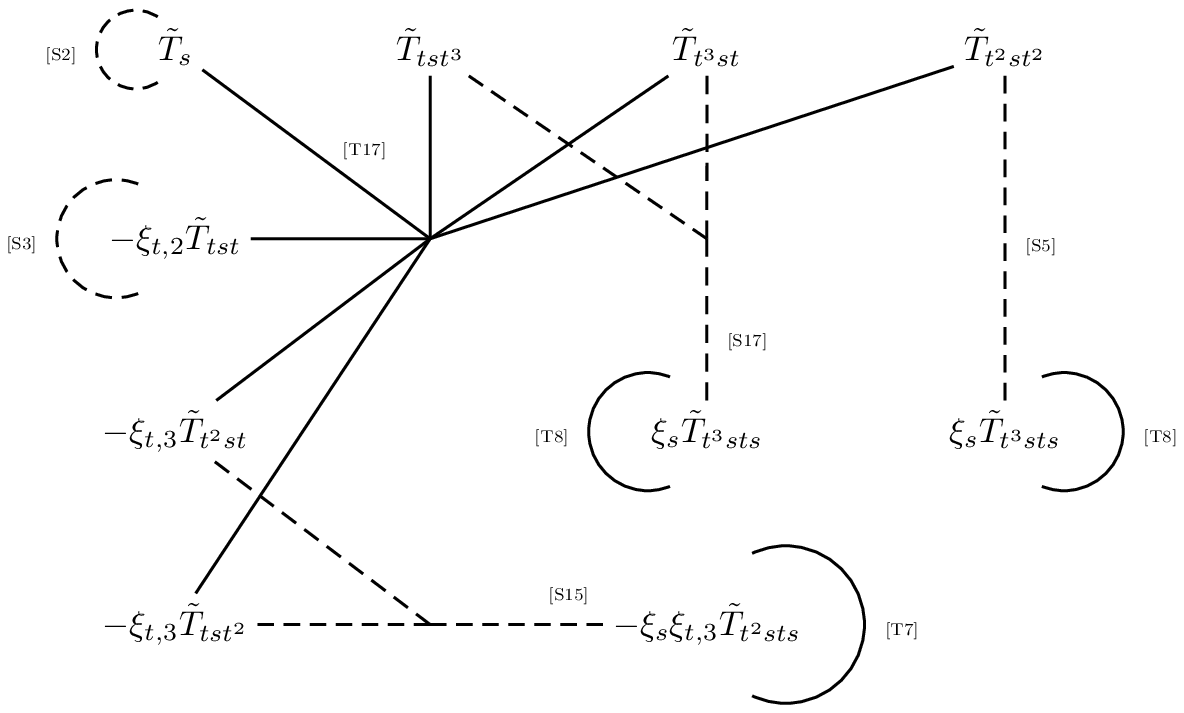}}
\caption{Diagram for the class element {[C11]} in the Hecke algebra of $G(4,1,2)$.  The solid edges indicate terms in the same $t$-class element (as labelled), and the dashed edges indicate terms in the same $s$-class element.}\label{fig:g412.c11}
\end{figure}


\end{document}